\documentclass{article}
\usepackage{graphicx} 
\usepackage{amsthm, amsmath, amssymb, tikz, bm}
\usepackage{mathtools}
\usepackage{xifthen}
\usepackage{comment}
\usepackage[margin=4.5cm]{geometry}
\usepackage[shortlabels]{enumitem}
\usepackage{todonotes}
\usepackage[colorlinks=true,
linkcolor=blue,citecolor=blue,
urlcolor=blue]{hyperref}

\usetikzlibrary{calc,shapes, backgrounds}
\usetikzlibrary[patterns]

\newtheorem{theorem}{Theorem}[section]
\newtheorem{lemma}[theorem]{Lemma}
\newtheorem{sublemma}{}[theorem]
\newtheorem{corollary}[theorem]{Corollary}

\newtheorem{conjecture}{Conjecture}
\newtheorem{proposition}[theorem]{Proposition}
\newtheorem{question}{Question}

\newcommand{\romannum}[1]{\romannumeral#1\relax}

\newcommand{\cC}{\mathcal{C}}

\newcommand{\cV}{\mathcal{V}}

\newcommand{\dd}{\backslash\backslash}

\makeatletter
\newcommand*{\rom}[1]{\expandafter\@slowromancap\romannumeral #1@}
\makeatother

\title{Super-minimally $3$-connected graphs}
\author{Wayne Ge}
\date{\today}

\begin{document}

\maketitle
\begin{abstract}
In this paper, we introduce super-minimally $k$-connected graphs, those $k$-connected graphs in which no proper subgraph is $k$-connected. For $k \geq 3$, this class lies strictly between the classes of minimally $k$-connected graphs and uniformly $k$-connected graphs. In particular, we determine the minimum number of degree-$3$ vertices in a super-minimally $3$-connected graph, thereby extending a result of Halin on minimally $3$-connected graphs. In addition, we determine the maximum number of edges in a super-minimally $3$-connected graph, extending Xu’s result for uniformly $3$-connected graphs, and providing an analogue of Halin's result for minimally $3$-connected graphs.

\end{abstract}

\section{Introduction}

Unless stated otherwise, all graph terminology in this paper follows Bondy and Murty~\cite{Bondy-Murty}. Throughout, all graphs considered are simple and finite unless specified otherwise. For a positive integer~$k$, a graph~$G$ is {\it minimally $k$-connected} if $G$ is $k$-connected, but the deletion $G \setminus e$ is not $k$-connected for every edge $e$ of $G$. Similarly, $G$ is {\it critically $k$-connected} if $G$ is $k$-connected, but $G - v$ is not $k$-connected for every vertex $v$ of $G$. In~\cite{BOP2002}, Beineke, Oellermann, and Pippert introduced the notion of {\it uniform connectivity}. A graph~$G$ is {\it uniformly $k$-connected} if, for every pair of distinct vertices, there are exactly $k$ internally disjoint paths connecting them. In this paper, we say that $G$ is {\it super-minimally $k$-connected} if $G$ is $k$-connected, but no proper subgraph of $G$ is $k$-connected. In Section~\ref{subsec:relations}, we formalize the relationships among the three connectivity classes and establish the following lemma describing their hierarchy.

\begin{lemma}\label{lem:inclusion lemma}
    For all integers $k$ exceeding one and all graphs $G$,
\begin{enumerate}
    \item[(\romannum{1})] if $G$ is uniformly $k$-connected, then $G$ is super-minimally $k$-connected, and
    \item[(\romannum{2})] if $G$ is super-minimally $k$-connected, then $G$ is minimally $k$-connected.
\end{enumerate}
Moreover, for all $k\geq3$, neither of the converses holds.
\end{lemma}

\subsection{Minimum number of degree-$k$ vertices}
Minimally $k$-connected graphs have received attention since the 1960s. A natural problem is to determine the minimum number of vertices of degree~$k$ in a $n$-vertex minimally $k$-connected graph. For the cases $k=2$ and $k=3$, these minima were determined, respectively, by Dirac~\cite[(6), (5)]{Dirac} and Halin~\cite[Satz~6]{Halin-German} in the following results.

\begin{theorem}\label{thm:Dirac_main}
    A minimally $2$-connected graph $G$ has at least $\frac{|V(G)| + 4}{3}$ vertices of degree two.
\end{theorem}

\begin{theorem}\label{thm:Halin_main}
    A minimally $3$-connected graph $G$ has at least $\frac{2|V(G)| + 6}{5}$ vertices of degree three.
\end{theorem}

It was conjectured by Halin~\cite{Halin} that, for all $k\geq 2$, there is a constant $c_k$ such that every minimally $k$-connected graph $G$ has at least $c_k|V(G)|$ vertices of degree $k$. Mader~\cite[Theorem 17]{Mader_eng} proved this conjecture by establishing the following generalization of Theorems~\ref{thm:Dirac_main} and~\ref{thm:Halin_main}.

\begin{theorem}\label{thm:Mader_main}
    For all $k \geq 2$, a minimally $k$-connected graph $G$ has at least $\frac{(k-1)|V(G)|+2k}{2k - 1}$ vertices of degree $k$.
\end{theorem}

When $k\geq 2$, since every super-minimally $k$-connected graph is also minimally $k$-connected, it is natural to ask the following.

\begin{question}\label{qs:main}
    Can Mader's lower bound be improved for super-minimally $k$-connect\-ed graphs?
\end{question}

For $k = 2$, we make the following elementary observation.

\begin{proposition}\label{prop:sm2c}
   Every super-minimally $2$-connected graph $G$ is a cycle and so has exactly $|V(G)|$ edges and  $|V(G)|$ vertices of degree two. 
\end{proposition}
    
For $k=3$, an analogous question for uniformly $3$-connected graphs has been answered by ~G\"{o}ring,~Hofmann, and~Streicher~\cite[Theorem 14]{GHS2022}.

\begin{theorem}\label{thm:uniformly 3-con}
     A uniformly $3$-connected graph $G$ has at least $\frac{2|V(G)| + 2}{3}$ vertices of degree three.
\end{theorem}

One of the main results of this paper is the following theorem, which answers the question in the affirmative for super-minimally $3$-connected graphs.

\begin{theorem}\label{thm:main}
    A super-minimally $3$-connected graph $G$ has at least $\frac{|V(G)| + 3}{2}$ vertices of degree three.
\end{theorem}

We demonstrate that the bound in Theorem~\ref{thm:main} is tight by constructing an infinite family of super-minimally $3$-connected graphs $G$ each having exactly $\frac{|V(G)| + 3}{2}$ vertices of degree three. Theorems~\ref{thm:uniformly 3-con} and~\ref{thm:main} verify the following conjecture when $k=3$.

\begin{conjecture}\label{conj:one}
    For every integer $k\geq 3$, there are constants $a_k$ and $b_k$ such that
    \begin{enumerate}
        \item[(\romannum{1})] every super-minimally $k$-connected graph has at least $a_k|V(G)|$ degree-$k$ vertices,
        \item[(\romannum{2})] every uniformly $k$-connected graph has at least $b_k|V(G)|$ degree-$k$ vertices, and
        \item[(\romannum{3})] $\frac{k-1}{2k-1}< a_k< b_k$.
    \end{enumerate}
    
\end{conjecture}

\subsection{Maximum number of edges}

Another natural problem is to determine the maximum number of edges in a minimally $k$-connected graph on $n$ vertices. For the cases $k=2$ and $k=3$, these maxima were determined by Dirac~\cite[(7)]{Dirac} and Halin~\cite[(7.6)]{Halin-density}, respectively.

\begin{theorem}
    If $G$ is a minimally $2$-connected graph with $|V(G)| \ge 4$, then
    \[
        |E(G)| \le 2|V(G)| - 4.
    \]
    Moreover, equality holds if and only if $G$ is isomorphic to $K_{2,\,|V(G)|-2}$.
\end{theorem}

\begin{theorem}\label{thm:m3c-density}
    If $G$ is a minimally $3$-connected graph with $|V(G)| \ge 7$, then
    \[
        |E(G)| \le 3|V(G)| - 9.
    \]
    Moreover, if $|V(G)| \ge 8$, equality holds if and only if $G$ is isomorphic to $K_{3,\,|V(G)|-3}$.
\end{theorem}

For $k=2$, Proposition~\ref{prop:sm2c} resolves the analogous problem for super-minimally $2$-connected graphs. Our second main result is the following analogue of Theorem~\ref{thm:m3c-density} for super-minimally $3$-connected graphs.

\begin{theorem}\label{thm:main-density}
If $G$ is a super-minimally $3$-connected graph, then
\[
|E(G)| \le 2|V(G)| - 2.
\]
Moreover, equality is attained if and only if $G$ is a wheel with at least three spokes.
\end{theorem}

It is straightforward to verify that every wheel with at least three spokes is uniformly $3$-connected. A direct consequence of Lemma~\ref{lem:inclusion lemma} and Theorem~\ref{thm:main-density} is the following result, originally proved by Xu~\cite[Theorem~3.2]{Xu2024}.

\begin{corollary}
If $G$ is a uniformly $3$-connected graph, then
\[
|E(G)| \le 2|V(G)| - 2.
\]
Moreover, equality is attained if and only if $G$ is a wheel with at least three spokes.
\end{corollary}

\section{Preliminaries}\label{sec:pre}

\subsection{Connectivity results for graphs}
A graph is {\it nontrivial} if it has more than one vertex. Let $G$ be a graph and let $v$ be a vertex of $G$. The {\it open neighborhood of $v$}, denoted by $N(v)$, is the set of all vertices adjacent to $v$; and the {\it closed neighborhood of $v$}, denoted by $N[v]$, is $N(v)\cup\{v\}$. A {\it vertex cut} of $G$ is a set $X$ of vertices such that $G-X$ has more connected components than $G$. Let $k$ be a nonnegative integer. A {\it $k$-separation} of $G$ is a pair of edge-disjoint subgraphs $\{G_1,G_2\}$ of $G$ for which $|V(G_1)\cap V(G_2)|=k$ such that $G_1\cup G_2=G$ and $\min\{|V(G_1)|,|V(G_2)|\}\geq k+1$. For a positive integer $k$, a graph $G$ is {\it $k$-connected} if $|V(G)| > k$ and $G$ has no vertex cut of size less than $k$. Equivalently, $G$ is {\it $k$-connected} if $|V(G)| > k$ and $G$ has no $k'$-separation for all $k' < k$. We note that, for some authors, the subgraphs $G_1$ and $G_2$ in a $k$-separation are required to have no isolated vertices. However, in the present paper, we relax this condition in order to simplify the case analysis in later proofs. For example, in Figure~\ref{fig_2-separation}, the subgraphs $G_1$ and $G_2$ form a $2$-separation of $G$ under this relaxed definition. As a consequence of this relaxation, we obtain the following result.

\begin{center}
    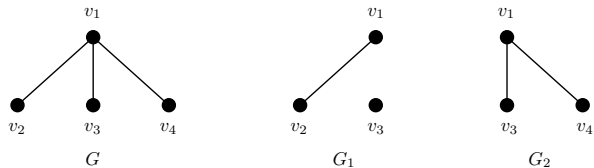
\begin{figure}[htb]
    \hbox to \hsize{
	\hfil
	\resizebox{8cm}{!}{\tikzset{every picture/.style={line width=0.75pt}} 

\begin{tikzpicture}[x=0.75pt,y=0.75pt,yscale=-1,xscale=1]

\draw  [fill={rgb, 255:red, 0; green, 0; blue, 0 }  ,fill opacity=1 ] (99.32,78.2) .. controls (99.32,75.62) and (101.42,73.52) .. (104,73.52) .. controls (106.59,73.52) and (108.68,75.62) .. (108.68,78.2) .. controls (108.68,80.78) and (106.59,82.88) .. (104,82.88) .. controls (101.42,82.88) and (99.32,80.78) .. (99.32,78.2) -- cycle ;
\draw    (49.5,127.2) -- (104,78.2) ;
\draw  [fill={rgb, 255:red, 0; green, 0; blue, 0 }  ,fill opacity=1 ] (153.82,127.2) .. controls (153.82,124.62) and (155.92,122.52) .. (158.5,122.52) .. controls (161.08,122.52) and (163.18,124.62) .. (163.18,127.2) .. controls (163.18,129.78) and (161.08,131.88) .. (158.5,131.88) .. controls (155.92,131.88) and (153.82,129.78) .. (153.82,127.2) -- cycle ;
\draw  [fill={rgb, 255:red, 0; green, 0; blue, 0 }  ,fill opacity=1 ] (99.32,127.2) .. controls (99.32,124.62) and (101.42,122.52) .. (104,122.52) .. controls (106.59,122.52) and (108.68,124.62) .. (108.68,127.2) .. controls (108.68,129.78) and (106.59,131.88) .. (104,131.88) .. controls (101.42,131.88) and (99.32,129.78) .. (99.32,127.2) -- cycle ;
\draw  [fill={rgb, 255:red, 0; green, 0; blue, 0 }  ,fill opacity=1 ] (44.82,127.2) .. controls (44.82,124.62) and (46.92,122.52) .. (49.5,122.52) .. controls (52.08,122.52) and (54.18,124.62) .. (54.18,127.2) .. controls (54.18,129.78) and (52.08,131.88) .. (49.5,131.88) .. controls (46.92,131.88) and (44.82,129.78) .. (44.82,127.2) -- cycle ;
\draw    (104,127.2) -- (104,78.2) ;
\draw    (158.5,127.2) -- (104,78.2) ;

\draw  [fill={rgb, 255:red, 0; green, 0; blue, 0 }  ,fill opacity=1 ] (397.32,78.2) .. controls (397.32,75.62) and (399.42,73.52) .. (402,73.52) .. controls (404.59,73.52) and (406.68,75.62) .. (406.68,78.2) .. controls (406.68,80.78) and (404.59,82.88) .. (402,82.88) .. controls (399.42,82.88) and (397.32,80.78) .. (397.32,78.2) -- cycle ;
\draw  [fill={rgb, 255:red, 0; green, 0; blue, 0 }  ,fill opacity=1 ] (451.82,127.2) .. controls (451.82,124.62) and (453.92,122.52) .. (456.5,122.52) .. controls (459.08,122.52) and (461.18,124.62) .. (461.18,127.2) .. controls (461.18,129.78) and (459.08,131.88) .. (456.5,131.88) .. controls (453.92,131.88) and (451.82,129.78) .. (451.82,127.2) -- cycle ;
\draw  [fill={rgb, 255:red, 0; green, 0; blue, 0 }  ,fill opacity=1 ] (397.32,127.2) .. controls (397.32,124.62) and (399.42,122.52) .. (402,122.52) .. controls (404.59,122.52) and (406.68,124.62) .. (406.68,127.2) .. controls (406.68,129.78) and (404.59,131.88) .. (402,131.88) .. controls (399.42,131.88) and (397.32,129.78) .. (397.32,127.2) -- cycle ;
\draw    (456.5,127.2) -- (402,78.2) ;

\draw    (402,127.2) -- (402,78.2) ;

\draw    (253,127.2) -- (307.5,78.2) ;
\draw  [fill={rgb, 255:red, 0; green, 0; blue, 0 }  ,fill opacity=1 ] (248.32,127.2) .. controls (248.32,124.62) and (250.42,122.52) .. (253,122.52) .. controls (255.59,122.52) and (257.68,124.62) .. (257.68,127.2) .. controls (257.68,129.78) and (255.59,131.88) .. (253,131.88) .. controls (250.42,131.88) and (248.32,129.78) .. (248.32,127.2) -- cycle ;
\draw  [fill={rgb, 255:red, 0; green, 0; blue, 0 }  ,fill opacity=1 ] (302.82,78.2) .. controls (302.82,75.62) and (304.92,73.52) .. (307.5,73.52) .. controls (310.09,73.52) and (312.18,75.62) .. (312.18,78.2) .. controls (312.18,80.78) and (310.09,82.88) .. (307.5,82.88) .. controls (304.92,82.88) and (302.82,80.78) .. (302.82,78.2) -- cycle ;
\draw  [fill={rgb, 255:red, 0; green, 0; blue, 0 }  ,fill opacity=1 ] (302.82,127.2) .. controls (302.82,124.62) and (304.92,122.52) .. (307.5,122.52) .. controls (310.09,122.52) and (312.18,124.62) .. (312.18,127.2) .. controls (312.18,129.78) and (310.09,131.88) .. (307.5,131.88) .. controls (304.92,131.88) and (302.82,129.78) .. (302.82,127.2) -- cycle ;

\draw (96,55.4) node [anchor=north west][inner sep=0.75pt]    {$v_{1}$};
\draw (96,138.6) node [anchor=north west][inner sep=0.75pt]    {$v_{3}$};
\draw (41.5,138.6) node [anchor=north west][inner sep=0.75pt]    {$v_{2}$};
\draw (150.5,138.6) node [anchor=north west][inner sep=0.75pt]    {$v_{4}$};
\draw (300,138.6) node [anchor=north west][inner sep=0.75pt]    {$v_{3}$};
\draw (244.5,138.6) node [anchor=north west][inner sep=0.75pt]    {$v_{2}$};
\draw (396,138.6) node [anchor=north west][inner sep=0.75pt]    {$v_{3}$};
\draw (450.5,138.6) node [anchor=north west][inner sep=0.75pt]    {$v_{4}$};
\draw (300,55.4) node [anchor=north west][inner sep=0.75pt]    {$v_{1}$};
\draw (394,55.4) node [anchor=north west][inner sep=0.75pt]    {$v_{1}$};
\draw (97,159.4) node [anchor=north west][inner sep=0.75pt]    {$G$};
\draw (275,159.4) node [anchor=north west][inner sep=0.75pt]    {$G_{1}$};
\draw (416,159.4) node [anchor=north west][inner sep=0.75pt]    {$G_{2}$};

\end{tikzpicture}}%
	\hfil
    }
    \caption{$\{G_1,G_2\}$ is a $2$-separation of $G.$}\label{fig_2-separation}
\end{figure}
\end{center}

\begin{proposition}
    For a positive integer $k$, if $G$ is a connected graph with at least $k+1$ vertices, then the following are equivalent.
\begin{enumerate}
    \item[(\romannum{1})] $G$ is not $k$-connected.
    \item[(\romannum{2})] $G$ has a $(k-1)$-separation.
    \item[(\romannum{3})] $G$ has a vertex cut of size $k-1$.
\end{enumerate}
\end{proposition}

In a graph $G$, let $x$ be a vertex and let $Y$ be a subset of $V(G)-\{x\}$. A family of $k$ internally disjoint $(x,Y)$-paths whose terminal vertices are distinct is called a {\it $k$-fan} from $x$ to $Y$.

\begin{lemma}\label{lem: cleaving is 3-con}
Let $G$ be a $3$-connected graph, and let $\{A,B\}$ be a $3$-separation of $G$ with $V(A)\cap V(B)=\{x,y,z\}$. Suppose $\min\{d_A(x),d_A(y),d_A(z)\}\ge 2$ and $G_A$ is a graph obtained from $A$ by adding a new vertex $u$ adjacent to each of $x,y,z$. Then $G_A$ is $3$-connected.
\end{lemma}

\begin{proof}
    Suppose $\{C,D\}$ is a $2$-separation of $G_A$. Without loss of generality, we may assume that $x\in V(C)-V(D)$. Suppose $u\in V(C)-V(D)$. Then $C$ contains all of $\{x,y,z\}$, and there is a vertex $a\in V(D)-V(C)$ such that $a\in V(A)-\{x,y,z\}$. By Menger's Theorem, we know that $A$ contains a $3$-fan from $a$ to $\{x,y,z\}$. Clearly, these paths are also contained in $G_A$. However, there are at most two internally disjoint paths from $a$ to distinct vertices of $C$, a contradiction. We deduce that $u\in V(C)\cap V(D)$, so $A$ is not $2$-connected. Let $\{E,F\}$ be a $1$-separation of $A$. Without loss of generality, we may assume that $x\in V(E)-V(F)$. Let $a$ be a vertex in $V(A)-\{x,y,z\}$. Suppose $a\in V(F)-V(E)$. Then $\{y,z\}\subseteq V(F)-V(E)$, otherwise there cannot be a $3$-fan from $a$ to $\{x,y,z\}$. However, since $d_A(x)\geq 2$, we know $x$ has a neighbor $b$ such that $b\in (V(A)-\{x,y,z\})\cap (V(E)-V(F))$. Since there cannot be a $2$-fan from $b$ to $\{y,z\}$, we obtain a contradiction. Thus, we conclude that, for all $a\in V(A)-\{x,y,z\}$, the vertex $a$ is contained in $V(E)$. Therefore, $V(F)=\{y,z\}$. However, we know $\min\{d_A(y),d_A(z)\}\geq 2$, a contradiction.
\end{proof}

A $2$-connected graph $G$ with at least four vertices is {\it internally $3$-connected} if, for every $2$-separation $\{G_1, G_2\}$ of $G$, one of $G_1$ or $G_2$ is isomorphic to a $3$-vertex path $P_3$. In particular, every $3$-connected graph is internally $3$-connected.

Let $v$ be a vertex of a graph $G$ with $d_G(v) \geq 4$. Suppose $(A, B)$ is a partition of $N(v)$ such that $\min\{|A|, |B|\} \geq 2$. Let $a$ and $b$ be two new vertices. The operation of {\it splitting $v$ into $a$ and $b$ with respect to $(A, B)$} consists of
\begin{enumerate}
    \item[(\romannum{1})] replacing $v$ by two adjacent vertices $a$ and $b$,
    \item[(\romannum{2})] joining $a$ to every vertex in $A$, and
    \item[(\romannum{3})] joining $b$ to every vertex in $B$.
\end{enumerate}
Our definition of vertex splitting is consistent with that of Tutte~\cite{Tutte1961}. Under this definition, if $G$ is a $3$-connected graph and $G'$ is obtained from $G$ by splitting a vertex, then $G'$ is also $3$-connected.

In the next lemma, although the graph $G$ is simple, when we contract edges, we do not require the resulting graph to be simplified.

\begin{lemma}\label{lem:contracting forest}
    Suppose $G$ is a graph with minimum degree at least three and $F$ is a forest in $G$. If $G/E(F)$ is simple and $3$-connected, then $G$ is $3$-connected.
\end{lemma}

\begin{proof}
    Let $E(F)=\{f_1,f_2,\dots,f_n\}$. We define a sequence of graphs as follows.
    \begin{enumerate}
        \item[(\romannum{1})] Let $G_0=G$.
        \item[(\romannum{2})] For $i\in\{1,2,\dots,n\}$, let $G_{i}=G_{i-1}/f_i$.
    \end{enumerate}
    Clearly, $G_n=G/E(F)$. Because $F$ is a forest, $f_i$ is not a loop in $G_{i-1}$ for all $i\in\{1,2,\dots,n\}$. Indeed, because $G_n$ is simple, we know $G_i$ is simple for all $i\in\{1,2,\dots,n\}$. Moreover, since the minimum degree of $G$ is at least three, the minimum degree of $G_i$ is at least three for all $i\in\{0,1,\dots,n\}$. Thus, $f_i$ is not incident to a vertex of degree less than three in $G_{i-1}$. Therefore, for $i\in\{0,1,\dots, n-1\}$, the graph $G_i$ can be obtained from $G_{i+1}$ by a vertex splitting. Because $G/E(F)$ is $3$-connected, we conclude that $G$ is also $3$-connected.
\end{proof}

\subsection{Minimally $3$-connected graphs}
The following lemmas for minimally $3$-connected graphs will be used frequently throughout the paper. Let $G$ be a graph and let $e$ be an edge that joins the vertices $x$ and $y$. The {\it order} of $e$ is the minimum of $d_G(x)$ and $d_G(y)$.

\begin{lemma}\cite[Chapter 1, Lemma 4.2]{Bollobas}\label{lem:bollobas}
    A $3$-connected graph is minimally $3$-connected if and only if there are exactly three internally disjoint $(x,y)$-paths for every pair of adjacent vertices $x$ and $y$.
\end{lemma}

\begin{lemma}\cite[Satz 5]{Halin-German}\label{lem:Halin_two_deg3_vtx}
    If $G$ is a minimally $3$-connected graph, then every cycle meets at least two vertices of degree three of $G$.
\end{lemma}

\begin{lemma}\cite[Satz 4]{Halin-German}\label{lem:Halin_contract_ord4}
    If $G$ is a minimally $3$-connected graph and $e$ is an edge of order at least four, then $G/e$ is minimally $3$-connected.
\end{lemma}

We note that in~\cite{Halin-German}, Halin considered only simple graphs. However, a minimally $3$-connected graph $G$ is clearly simple. Moreover, by Lemma~\ref{lem:Halin_two_deg3_vtx}, no triangle in $G$ contains an edge of order at least four. Therefore, if $e$ is an edge of order at least four in $G$, then the contraction $G/e$ is always simple.

\begin{lemma}\cite[Korollar 1]{Mader}\label{lem:Mader_induce_forest}
    The subgraph induced by vertices of degree greater than $k$ in a minimally $k$-connected graph is a forest.
\end{lemma} 

\subsection{Super-minimally $3$-connected graphs}\label{subsec:relations}
Evidently, a super-minimally $k$-connected graph is both minimally $k$-connected and critically $k$-connected. However, as we will show, the converse does not hold. The {\it length} of a path $P$ is the number of edges in $P$. Let $k$ be an integer exceeding one, a {\it theta graph} $\Theta(l_1, l_2, \dots, l_k)$ consists of two distinct vertices connected by $k$ internally disjoint paths of lengths $l_1, l_2, \dots, l_k$ such that at most one of $l_1, l_2, \dots, l_k$ equals one. A {\it $k$-dimensional wheel} $W(l_1, l_2, \dots, l_k)$ is obtained from $\Theta(l_1, l_2, \dots, l_k)$ by adding a vertex adjacent to all degree-$2$ vertices in $\Theta(l_1, l_2, \dots, l_k)$. An {\it augmented $k$-dimensional wheel} $W^+(l_1, l_2, \dots, l_k)$ is formed by adding a vertex adjacent to all vertices in $\Theta(l_1, l_2, \dots, l_k)$. It is straightforward to verify that the $4$-dimensional wheel $W(3,3,3,3)$ (see Figure~\ref{fig_multi_wheel}(a)) is both minimally $3$-connected and critically $3$-connected. However, it contains a $3$-connected proper subgraph isomorphic to the $3$-dimensional wheel $W(3,3,3)$ (see Figure~\ref{fig_multi_wheel}(b)), and thus it fails to be super-minimally $3$-connected. 
\begin{center}
    \begin{figure}[htb]
    \hbox to \hsize{
	\hfil
	\resizebox{8cm}{!}{\input{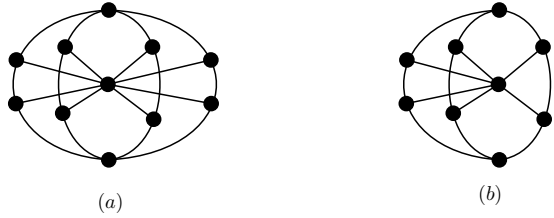}}%
	\hfil
    }
    \caption{(a). A $4$-dimensional wheel $W(3,3,3,3)$, and (b). a $3$-connected proper subgraph.}\label{fig_multi_wheel}
\end{figure}
\end{center}

On the other hand, a super-minimally $3$-connected graph is not necessarily uniformly $3$-connected. From a cycle with vertices $v_1, v_2, \dots, v_{2n}$ labeled in clockwise order, an {\it alternating double wheel} $A_n$ is obtained by adding two vertices $x$ and $y$ such that $x$ is adjacent to $v_1, v_3, \dots, v_{2n-1}$, and $y$ is adjacent to $v_2, v_4, \dots, v_{2n}$. Figure~\ref{fig_TW_4} shows the alternating double wheel $A_4$. It is straightforward to verify that, for all $n \geq 4$, the graph $A_n$ is super-minimally $3$-connected. It is not uniformly $3$-connected, since there are $n$ internally disjoint paths between $x$ and $y$.

\begin{center}
    \begin{figure}[htb]
    \hbox to \hsize{
	\hfil
	\resizebox{4cm}{!}{\tikzset{every picture/.style={line width=0.75pt}} 

\begin{tikzpicture}[x=0.75pt,y=0.75pt,yscale=-1,xscale=1]

\draw   (176,140) .. controls (176,103.55) and (205.55,74) .. (242,74) .. controls (278.45,74) and (308,103.55) .. (308,140) .. controls (308,176.45) and (278.45,206) .. (242,206) .. controls (205.55,206) and (176,176.45) .. (176,140) -- cycle ;
\draw  [fill={rgb, 255:red, 0; green, 0; blue, 0 }  ,fill opacity=1 ] (237.32,206) .. controls (237.32,203.42) and (239.42,201.32) .. (242,201.32) .. controls (244.58,201.32) and (246.68,203.42) .. (246.68,206) .. controls (246.68,208.58) and (244.58,210.68) .. (242,210.68) .. controls (239.42,210.68) and (237.32,208.58) .. (237.32,206) -- cycle ;
\draw  [fill={rgb, 255:red, 0; green, 0; blue, 0 }  ,fill opacity=1 ] (237.32,74) .. controls (237.32,71.42) and (239.42,69.32) .. (242,69.32) .. controls (244.58,69.32) and (246.68,71.42) .. (246.68,74) .. controls (246.68,76.58) and (244.58,78.68) .. (242,78.68) .. controls (239.42,78.68) and (237.32,76.58) .. (237.32,74) -- cycle ;
\draw  [fill={rgb, 255:red, 0; green, 0; blue, 0 }  ,fill opacity=1 ] (303.32,140) .. controls (303.32,137.42) and (305.42,135.32) .. (308,135.32) .. controls (310.58,135.32) and (312.68,137.42) .. (312.68,140) .. controls (312.68,142.58) and (310.58,144.68) .. (308,144.68) .. controls (305.42,144.68) and (303.32,142.58) .. (303.32,140) -- cycle ;
\draw  [fill={rgb, 255:red, 0; green, 0; blue, 0 }  ,fill opacity=1 ] (171.32,140) .. controls (171.32,137.42) and (173.42,135.32) .. (176,135.32) .. controls (178.58,135.32) and (180.68,137.42) .. (180.68,140) .. controls (180.68,142.58) and (178.58,144.68) .. (176,144.68) .. controls (173.42,144.68) and (171.32,142.58) .. (171.32,140) -- cycle ;
\draw  [fill={rgb, 255:red, 0; green, 0; blue, 0 }  ,fill opacity=1 ] (191.82,188.2) .. controls (191.82,185.62) and (193.92,183.52) .. (196.5,183.52) .. controls (199.08,183.52) and (201.18,185.62) .. (201.18,188.2) .. controls (201.18,190.78) and (199.08,192.88) .. (196.5,192.88) .. controls (193.92,192.88) and (191.82,190.78) .. (191.82,188.2) -- cycle ;
\draw  [fill={rgb, 255:red, 0; green, 0; blue, 0 }  ,fill opacity=1 ] (191.82,91.2) .. controls (191.82,88.62) and (193.92,86.52) .. (196.5,86.52) .. controls (199.08,86.52) and (201.18,88.62) .. (201.18,91.2) .. controls (201.18,93.78) and (199.08,95.88) .. (196.5,95.88) .. controls (193.92,95.88) and (191.82,93.78) .. (191.82,91.2) -- cycle ;

\draw  [fill={rgb, 255:red, 0; green, 0; blue, 0 }  ,fill opacity=1 ] (281.82,189.2) .. controls (281.82,186.62) and (283.92,184.52) .. (286.5,184.52) .. controls (289.08,184.52) and (291.18,186.62) .. (291.18,189.2) .. controls (291.18,191.78) and (289.08,193.88) .. (286.5,193.88) .. controls (283.92,193.88) and (281.82,191.78) .. (281.82,189.2) -- cycle ;
\draw  [fill={rgb, 255:red, 0; green, 0; blue, 0 }  ,fill opacity=1 ] (281.82,92.2) .. controls (281.82,89.62) and (283.92,87.52) .. (286.5,87.52) .. controls (289.08,87.52) and (291.18,89.62) .. (291.18,92.2) .. controls (291.18,94.78) and (289.08,96.88) .. (286.5,96.88) .. controls (283.92,96.88) and (281.82,94.78) .. (281.82,92.2) -- cycle ;

\draw  [fill={rgb, 255:red, 0; green, 0; blue, 0 }  ,fill opacity=1 ] (237.32,140) .. controls (237.32,137.42) and (239.42,135.32) .. (242,135.32) .. controls (244.58,135.32) and (246.68,137.42) .. (246.68,140) .. controls (246.68,142.58) and (244.58,144.68) .. (242,144.68) .. controls (239.42,144.68) and (237.32,142.58) .. (237.32,140) -- cycle ;
\draw  [fill={rgb, 255:red, 0; green, 0; blue, 0 }  ,fill opacity=1 ] (377.32,140) .. controls (377.32,137.42) and (379.42,135.32) .. (382,135.32) .. controls (384.59,135.32) and (386.68,137.42) .. (386.68,140) .. controls (386.68,142.58) and (384.59,144.68) .. (382,144.68) .. controls (379.42,144.68) and (377.32,142.58) .. (377.32,140) -- cycle ;
\draw    (242,74) -- (242,140) ;
\draw    (242,140) -- (242,206) ;
\draw    (176,140) -- (242,140) ;
\draw    (242,140) -- (308,140) ;
\draw    (286.5,92.2) -- (382,140) ;
\draw    (286.5,189.2) -- (382,140) ;
\draw    (196.5,91.2) .. controls (193.07,22.27) and (381.07,20.27) .. (382,140) ;
\draw    (196.5,188.2) .. controls (191.73,260.27) and (377.73,265.6) .. (382,140) ;

\draw (237,58.4) node [anchor=north west][inner sep=0.75pt]    {$v_{1}$};
\draw (292,81.4) node [anchor=north west][inner sep=0.75pt]    {$v_{2}$};
\draw (315,138.4) node [anchor=north west][inner sep=0.75pt]    {$v_{3}$};
\draw (288.5,202.6) node [anchor=north west][inner sep=0.75pt]    {$v_{4}$};
\draw (237,219.4) node [anchor=north west][inner sep=0.75pt]    {$v_{5}$};
\draw (181,198.4) node [anchor=north west][inner sep=0.75pt]    {$v_{6}$};
\draw (153,140.4) node [anchor=north west][inner sep=0.75pt]    {$v_{7}$};
\draw (176,78.4) node [anchor=north west][inner sep=0.75pt]    {$v_{8}$};
\draw (244,148.4) node [anchor=north west][inner sep=0.75pt]    {$x$};
\draw (391,135.73) node [anchor=north west][inner sep=0.75pt]    {$y$};

\end{tikzpicture}}%
	\hfil
    }
    \caption{An alternating double wheel $A_4$.}\label{fig_TW_4}
\end{figure}
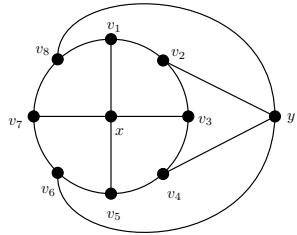
\end{center}

For all $k \geq 4$, there are also super-minimally $k$-connected graphs that are not uniformly $k$-connected. By $K_n - C_n$, we mean the graph obtained from the complete graph $K_n$ by deleting the edges of a Hamiltonian cycle $C_n$. Similarly, by $K_n-P_n$, we mean the graph obtained from $K_n$ by deleting the edges of a Hamiltonian path $P_n$.

\begin{proposition}\label{prop:K_n-C_n is (n-3)-con}
    For all $n\geq 5$, the graph $K_n-C_n$ is $(n-3)$-connected.
\end{proposition}

\begin{proof}
    We argue by induction on $n$. When $n=5$, it is straightforward to verify that $K_5-C_5$ is isomorphic to $C_5$ and hence is $2$-connected, so the base case holds. Now we may assume that the statement holds for all $n'$ such that $5\leq n'<n$. For all $v\in V(K_n-C_n)$, we first observe that
    \[(K_n-C_n)-v= K_{n-1}-P_{n-1},\]
    and hence $(K_n-C_n)-v$ contains $K_{n-1}-C_{n-1}$ as a spanning subgraph. By the inductive hypothesis, $(K_n-C_n)-v$ is $(n-4)$-connected for all $v\in V(K_n-C_n)$ and thus $K_n-C_n$ is $(n-3)$-connected.  
\end{proof}

Now we define $Q_n$ to be the graph that is obtained from $K_n - C_n$ by adding two nonadjacent vertices $x$ and $y$, where each of $x$ and $y$ is adjacent to all the vertices in $K_n - C_n$. Figure~\ref{fig_Q_5} shows the graph $Q_5$. For $k\geq 4$, using Proposition~\ref{prop:K_n-C_n is (n-3)-con}, one can easily verify that $Q_{k+1}$ is $k$-connected. Moreover, $Q_{k+1}$ is super-minimally $k$-connected. To see this, suppose that $H$ is a proper $k$-connected subgraph of $Q_{k+1}$. Observe that every edge $e$ in $Q_{k+1}$ is adjacent to a vertex $v$ of degree $k$. If $e\notin E(H)$, then $v\notin V(H)$. Indeed, $Q_{k+1}$ has a path that contains all of the vertices of degree $k$. The deletion of $v$ will force the deletion of all of the vertices of degree $k$. Once all degree-$k$ vertices are deleted, the remaining graph consists of two isolated vertices. Thus, $Q_{k+1}$ does not have a proper $k$-connected subgraph and hence is super-minimally $k$-connected. However, $Q_{k+1}$ is not uniformly $k$-connected as there are $k+1$ internally disjoint paths between $x$ and $y$. Therefore, for all $k \geq 4$, the graph $Q_{k+1}$ provides an example of a super-minimally $k$-connected graph that is not uniformly $k$-connected.

\begin{center}
    \begin{figure}[htb]
    \hbox to \hsize{
	\hfil
	\resizebox{5cm}{!}{\tikzset{every picture/.style={line width=0.75pt}} 

\begin{tikzpicture}[x=0.75pt,y=0.75pt,yscale=-1,xscale=1]

\draw  [fill={rgb, 255:red, 0; green, 0; blue, 0 }  ,fill opacity=1 ] (273.32,112.5) .. controls (273.32,109.92) and (275.42,107.82) .. (278,107.82) .. controls (280.58,107.82) and (282.68,109.92) .. (282.68,112.5) .. controls (282.68,115.08) and (280.58,117.18) .. (278,117.18) .. controls (275.42,117.18) and (273.32,115.08) .. (273.32,112.5) -- cycle ;
\draw  [fill={rgb, 255:red, 0; green, 0; blue, 0 }  ,fill opacity=1 ] (228.75,51.16) .. controls (228.75,48.57) and (230.85,46.48) .. (233.43,46.48) .. controls (236.02,46.48) and (238.11,48.57) .. (238.11,51.16) .. controls (238.11,53.74) and (236.02,55.84) .. (233.43,55.84) .. controls (230.85,55.84) and (228.75,53.74) .. (228.75,51.16) -- cycle ;
\draw  [fill={rgb, 255:red, 0; green, 0; blue, 0 }  ,fill opacity=1 ] (156.64,74.59) .. controls (156.64,72) and (158.73,69.91) .. (161.32,69.91) .. controls (163.9,69.91) and (166,72) .. (166,74.59) .. controls (166,77.17) and (163.9,79.27) .. (161.32,79.27) .. controls (158.73,79.27) and (156.64,77.17) .. (156.64,74.59) -- cycle ;
\draw  [fill={rgb, 255:red, 0; green, 0; blue, 0 }  ,fill opacity=1 ] (156.64,150.41) .. controls (156.64,147.83) and (158.73,145.73) .. (161.32,145.73) .. controls (163.9,145.73) and (166,147.83) .. (166,150.41) .. controls (166,153) and (163.9,155.09) .. (161.32,155.09) .. controls (158.73,155.09) and (156.64,153) .. (156.64,150.41) -- cycle ;
\draw  [fill={rgb, 255:red, 0; green, 0; blue, 0 }  ,fill opacity=1 ] (228.75,173.84) .. controls (228.75,171.26) and (230.85,169.16) .. (233.43,169.16) .. controls (236.02,169.16) and (238.11,171.26) .. (238.11,173.84) .. controls (238.11,176.43) and (236.02,178.52) .. (233.43,178.52) .. controls (230.85,178.52) and (228.75,176.43) .. (228.75,173.84) -- cycle ;
\draw    (233.43,51.16) -- (233.43,173.84) ;
\draw    (161.32,150.41) -- (278,112.5) ;
\draw    (161.32,74.59) -- (278,112.5) ;
\draw    (161.32,74.59) -- (233.43,173.84) ;
\draw    (233.43,51.16) -- (161.32,150.41) ;
\draw  [fill={rgb, 255:red, 0; green, 0; blue, 0 }  ,fill opacity=1 ] (109.64,118.41) .. controls (109.64,115.83) and (111.73,113.73) .. (114.32,113.73) .. controls (116.9,113.73) and (119,115.83) .. (119,118.41) .. controls (119,121) and (116.9,123.09) .. (114.32,123.09) .. controls (111.73,123.09) and (109.64,121) .. (109.64,118.41) -- cycle ;
\draw  [fill={rgb, 255:red, 0; green, 0; blue, 0 }  ,fill opacity=1 ] (286.64,62.41) .. controls (286.64,59.83) and (288.73,57.73) .. (291.32,57.73) .. controls (293.9,57.73) and (296,59.83) .. (296,62.41) .. controls (296,65) and (293.9,67.09) .. (291.32,67.09) .. controls (288.73,67.09) and (286.64,65) .. (286.64,62.41) -- cycle ;
\draw    (114.32,118.41) -- (161.32,74.59) ;
\draw    (114.32,118.41) -- (233.43,51.16) ;
\draw    (114.32,118.41) -- (278,112.5) ;
\draw    (161.32,150.41) -- (114.32,118.41) ;
\draw    (114.32,118.41) -- (233.43,173.84) ;
\draw    (233.43,51.16) -- (291.32,62.41) ;
\draw    (278,112.5) -- (291.32,62.41) ;
\draw    (161.32,150.41) -- (291.32,62.41) ;
\draw    (233.43,173.84) -- (291.32,62.41) ;
\draw    (161.32,74.59) -- (291.32,62.41) ;
\draw  [dash pattern={on 4.5pt off 4.5pt}] (129.32,108) .. controls (129.32,63.82) and (165.14,28) .. (209.32,28) .. controls (253.5,28) and (289.32,63.82) .. (289.32,108) .. controls (289.32,152.18) and (253.5,188) .. (209.32,188) .. controls (165.14,188) and (129.32,152.18) .. (129.32,108) -- cycle ;

\draw (93,108.4) node [anchor=north west][inner sep=0.75pt]    {$x$};
\draw (302,53.4) node [anchor=north west][inner sep=0.75pt]    {$y$};
\draw (163,43.4) node [anchor=north west][inner sep=0.75pt]  [font=\small]  {$K_{5} -C_{5}$};

\end{tikzpicture}}%
	\hfil
    }
    \caption{$Q_5$ is super-minimally $4$-connected, but not uniformly $4$-connected.}\label{fig_Q_5}
\end{figure}
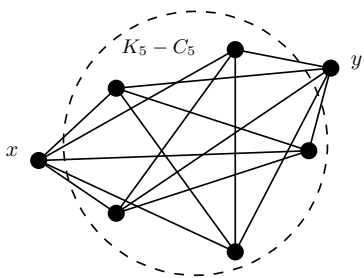
\end{center}

\noindent{\it Proof of Lemma~\ref{lem:inclusion lemma}.}
\setcounter{theorem}{6}
For all $k\geq 2$, it is clear that a super-minimally $k$-connected graph is also minimally $k$-connected. To prove (\romannum{1}), we suppose that $G$ is a uniformly $k$-connected graph and $H$ is a proper $k$-connected subgraph of $G$. First we show the following.

\begin{sublemma}
    $H$ is not a spanning subgraph of $G$.
\end{sublemma}

Assume that $H$ is a spanning subgraph of $G$. Then there is an edge $e$ in $E(G)-E(H)$ that joins two vertices $a$ and $b$ in $H$. Clearly, there are at least $k+1$ internally disjoint paths between $a$ and $b$, a contradiction.\\

Therefore, there is a vertex $v\in V(G)-V(H)$. Since $|V(H)|\geq k+1$, by Menger's Theorem, there are two internally disjoint paths $P_1,P_2$ joining $v$ to distinct vertices $u_1$ and $u_2$, respectively, such that $V(P_1)\cap V(H)=\{u_1\}$ and $V(P_2)\cap V(H)=\{u_2\}$. Therefore, there are at least $k+1$ internally disjoint $u_1u_2$-paths in $G$, consisting of $k$ such paths contained in $H$ and one formed by $P_1\cup P_2$. Thus $G$ is not uniformly $k$-connected, a contradiction. We conclude that $G$ does not have a proper $k$-connected subgraph, so $G$ is super-minimally $k$-connected.

When $k \geq 3$, it is clear that, for all $i\geq 1$, the complete bipartite graph $K_{k,k+i}$ is minimally, but not super-minimally, $k$-connected. This shows that the converse of (\romannum{2}) fails. Moreover, the failure of the converse of (\romannum{1}) is demonstrated by the examples provided earlier in this section.
\qed\\

We note that the condition $k \geq 2$ in Lemma~\ref{lem:inclusion lemma} is necessary. For $k = 1$, it is elementary to observe that a graph is minimally $1$-connected if and only if it is a nontrivial tree. Beineke, Oellermann, and Pippert~\cite{BOP2002} observed that the uniformly $1$-connected graphs are also nontrivial trees. However, the only super-minimally $1$-connected graph is the two-vertex graph $K_2$, so (\romannum{1}) in Lemma~\ref{lem:inclusion lemma} fails in this case. Moreover, when $k = 2$, the converse of (\romannum{2}) in Lemma~\ref{lem:inclusion lemma} does not hold. But the converse of (\romannum{1}) holds. Beineke, Oellermann, and Pippert~\cite{BOP2002} observed that the only uniformly $2$-connected graphs are cycles, and, by Proposition~\ref{prop:sm2c}, these are the only super-minimally $2$-connected graphs. Figure~\ref{fig_venn} presents a Venn diagram illustrating the relationships among super-minimally, minimally, critically, and uniformly $3$-connected graphs, along with representative examples in each region. 

\begin{center}
    \begin{figure}[htb]
    \hbox to \hsize{
	\hfil
	\resizebox{12.5cm}{!}{\tikzset{every picture/.style={line width=0.75pt}} 

\begin{tikzpicture}[x=0.75pt,y=0.75pt,yscale=-1,xscale=1]

\draw   (38,19) -- (606,19) -- (606,384.5) -- (38,384.5) -- cycle ;
\draw   (58,193.75) .. controls (58,112.15) and (146.2,46) .. (255,46) .. controls (363.8,46) and (452,112.15) .. (452,193.75) .. controls (452,275.35) and (363.8,341.5) .. (255,341.5) .. controls (146.2,341.5) and (58,275.35) .. (58,193.75) -- cycle ;
\draw   (201,193.75) .. controls (201,112.15) and (289.2,46) .. (398,46) .. controls (506.8,46) and (595,112.15) .. (595,193.75) .. controls (595,275.35) and (506.8,341.5) .. (398,341.5) .. controls (289.2,341.5) and (201,275.35) .. (201,193.75) -- cycle ;

\draw   (209.5,177.5) .. controls (209.5,128.62) and (261.88,89) .. (326.5,89) .. controls (391.12,89) and (443.5,128.62) .. (443.5,177.5) .. controls (443.5,226.38) and (391.12,266) .. (326.5,266) .. controls (261.88,266) and (209.5,226.38) .. (209.5,177.5) -- cycle ;
\draw   (225.5,150) .. controls (225.5,129.01) and (270.72,112) .. (326.5,112) .. controls (382.28,112) and (427.5,129.01) .. (427.5,150) .. controls (427.5,170.99) and (382.28,188) .. (326.5,188) .. controls (270.72,188) and (225.5,170.99) .. (225.5,150) -- cycle ;

\draw (276.5,353) node [anchor=north west][inner sep=0.75pt]   [align=left] {$\displaystyle 3$-connected: $\displaystyle K_{5}$};
\draw (78,166.62) node [anchor=north west][inner sep=0.75pt]   [align=left] {minimally \\$\displaystyle 3$-connected: $\displaystyle K_{3,4}$};
\draw (472,161.8) node [anchor=north west][inner sep=0.75pt]   [align=left] {critically \\$\displaystyle 3$-connected: \\$\displaystyle W^{+}$(3,3,3)};
\draw (265.5,195.75) node [anchor=north west][inner sep=0.75pt]   [align=left] {super-minimally \\$\displaystyle 3$-connected: $\displaystyle A_{4}{}$};
\draw (289.5,285.17) node [anchor=north west][inner sep=0.75pt]    {$W( 3,3,3,3)$};
\draw (237,141) node [anchor=north west][inner sep=0.75pt]   [align=left] {uniformly $\displaystyle 3$-connected: $\displaystyle K_{3,3}$$ $};

\end{tikzpicture}}%
	\hfil
    }
    \caption{A Venn diagram of super-minimally, minimally, critically, and uniformly $3$-connected graphs.}\label{fig_venn}
\end{figure}
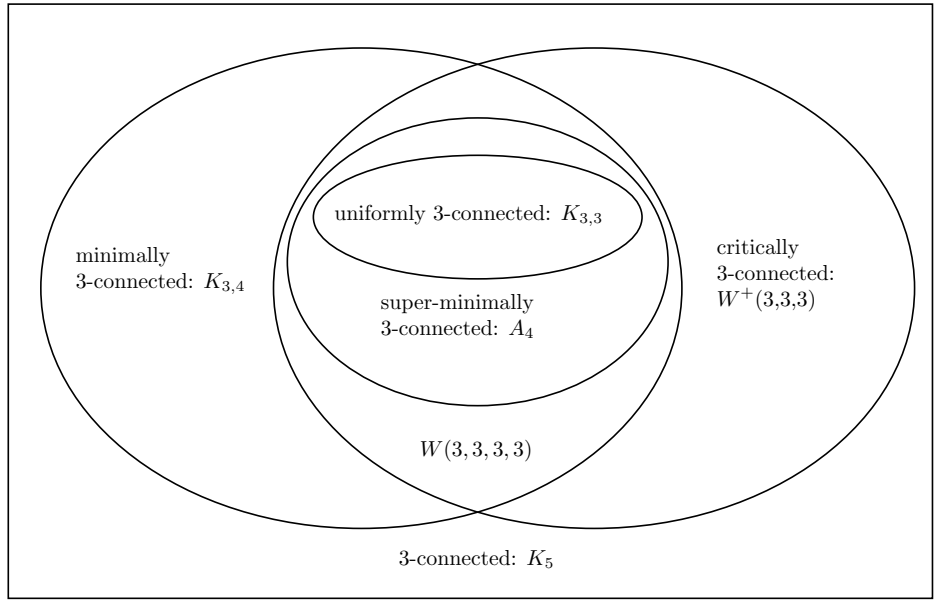
\end{center}

\begin{lemma}\label{lem:m3c and c3c}
    Let $G$ be a minimally and critically $3$-connected graph and $F$ be the subgraph induced by the vertices of degree greater than three. Then $G/E(F)$ is both minimally and critically $3$-connected. In particular, $G/E(F)$ is simple.
\end{lemma}

\begin{proof}
    Let $H=G/E(F)$. By repeated application of Lemma~\ref{lem:Halin_contract_ord4}, we deduce that $H$ is minimally $3$-connected, so $H$ is simple. Let $V_3(G)$ be the set of degree-$3$ vertices in $G$ and let $V_3(H)$ be the set of degree-$3$ vertices in $H$. Since each degree-$3$ vertex in $G$ is only incident to edges in $E(G)-E(F)$, we know $V_3(G)=V_3(H)$. Moreover, each edge in $H$ is incident to at least one vertex in $V_3(H)$. Hence, in $H$, the set $A$ of vertices of degree greater than three forms a stable set. Suppose that there is a vertex $v$ of $H$ such that $H-v$ is $3$-connected. Observe that $v$ cannot have a neighbor of degree three. Therefore, $N_H(v)\subseteq A$ and $v$ is not incident to degree-$3$ vertices in $G$. Thus the minimum degree of $G-v$ is at least three. Since $H-v=(G-v)/E(F)$, and $H-v$ is simple and $3$-connected, by Lemma~\ref{lem:contracting forest}, $G-v$ is $3$-connected, a contradiction. We conclude that such a vertex $v$ does not exist and $H$ is critically $3$-connected.
\end{proof}

Since a super-minimally $3$-connected graph is both minimally $3$-connected and critically $3$-connected, the next result follows immediately from Lemma~\ref{lem:m3c and c3c}.

\begin{corollary}
    Let $G$ be a super-minimally $3$-connected graph and $F$ be the subgraph induced by the vertices of degree greater than three. Then $G/E(F)$ is both minimally and critically $3$-connected. In particular, $G/E(F)$ is simple.
\end{corollary}

We note that if $G$ is super-minimally $3$-connected and $e$ is an edge of order at least four in $G$, the contraction $G/e$ is not necessarily super-minimally $3$-connected. Consider the graph $G$ and the edge $e$ shown in Figure~\ref{fig_not_sm3c}(a). One can verify that $G$ is super-minimally $3$-connected and that $e$ is an edge of order four. However, the contraction $G/e$ (see Figure~\ref{fig_not_sm3c}(b)) is not super-minimally $3$-connected, since the subgraph obtained by deleting the two white vertices remains $3$-connected.

\begin{center}
    \begin{figure}[htb]
    \hbox to \hsize{
	\hfil
	\resizebox{10cm}{!}{\input{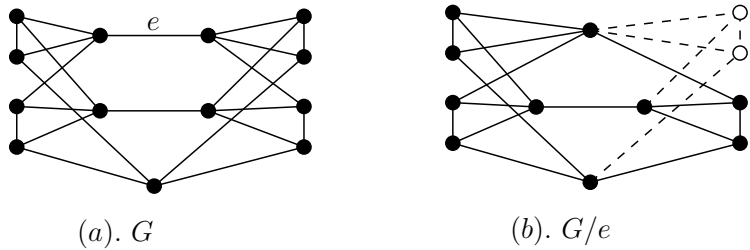}}%
	\hfil
    }
    \caption{$G$ is a super-minimally $3$-connected graph and $e$ is an edge of order four. However, $G/e$ is not super-minimally $3$-connected.}\label{fig_not_sm3c}
\end{figure}
\end{center}

\section{Lemmas for bipartite graphs}

In this section, we establish several lemmas concerning bipartite graphs.

\begin{lemma}\label{lem:bipartite general lemma}
Let $G$ be a bipartite graph with bipartition $(X,Y)$ such that $|Y| \ge 2|X| - k$, where $k$ is an integer, 
and $d_G(y) \ge 3$ for all $y \in Y$. 
Suppose that $\{A,B\}$ is a $2$-separation of $G$, and let $S = Y \cap V(A) \cap V(B)$. 
Then $G$ has a subgraph $G' \in \{A-S,\, B-S\}$ with bipartition $(X',Y')$ such that $|X'| < |X|$ and $|Y'| \ge 2|X'| - \lfloor \tfrac{k+4}{2} \rfloor$. 
Moreover, $d_G(y) = d_{G'}(y)$ for all $y \in Y'$.
\end{lemma}

\begin{proof}
Let $X_A = X \cap V(A)$ and $X_B = X \cap V(B)$. Suppose that $X \cap (V(A)-V(B)) = \emptyset$. Then there exists a vertex $y \in Y \cap (V(A)-V(B))$. Since $d_A(y) = d_G(y) \ge 3$, at least one neighbor of $y$ lies in $X \cap (V(A)-V(B))$, a contradiction. By symmetry, it follows that $X \cap (V(B)-V(A)) \ne \emptyset$. Thus we conclude that 
\begin{sublemma}\label{sublem:size smaller}
    $|X_A| < |X|$ and $|X_B| < |X|$.
\end{sublemma}

For each vertex $y \in Y-S$, either $N_G(y) \subseteq X_A$ or $N_G(y) \subseteq X_B$, but not both. Define $Y_A=\{y\in Y-S:N_G(y)\subseteq X_A\}$ and $Y_B=\{y\in Y-S:N_G(y)\subseteq X_B\}$. Suppose $|S| = m$. Clearly $m \in \{0,1,2\}$, and we have $|X_A| + |X_B| = |X| + 2 - m$. Therefore,
\begin{align*}
    |Y_A| + |Y_B| &= |Y| - m \\
                   &\ge 2|X| - k - m \\
                   &= 2|X_A| + 2|X_B| - k - 4 + m \\
                   &\ge 2|X_A| + 2|X_B| - k - 4.
\end{align*}
It follows that either $|Y_A| \ge 2|X_A| - \lfloor \tfrac{k+4}{2} \rfloor$ or $|Y_B| \ge 2|X_B| - \lfloor \tfrac{k+4}{2} \rfloor$. Without loss of generality, assume that $|Y_A| \ge 2|X_A| - \lfloor \tfrac{k+4}{2} \rfloor$. Let $G'$ be the graph $A-S = G[X_A \cup Y_A]$. Then $G'$ is a bipartite graph with bipartition $(X_A,Y_A)$ satisfying $|Y_A| \ge 2|X_A| - \lfloor \tfrac{k+4}{2} \rfloor$. By~\ref{sublem:size smaller}, we have $|X_A| < |X|$. Moreover, since $N_G(y) \subseteq X_A$, for every $y \in Y_A$, it follows that $d_G(y) = d_{G'}(y)$ for all $y \in Y_A$.
\end{proof}

A bipartite graph $G$ with bipartition $(X,Y)$ is {\it semi-cubic} if $d_G(y)=3$ for all $y\in Y$. When $G$ is a bipartite graph with a fixed bipartition $(X,Y)$, we say $H$ is a {\it $K_{s,t}$-subgraph of $G$} if $H$ is isomorphic to $K_{s,t}$ such that $|V(H)\cap X|=s$ and $|V(H)\cap Y|=t$. Let $S$ be a set of vertices of $G$. The {\it open neighborhood $N(S)$ of $S$} is $\cup_{s\in S}N(s)-S$, and the {\it closed neighborhood $N[S]$ of $S$} is $\cup_{s\in S}N[s]$.

\begin{lemma}\label{lem:2n-4}
    Let $G$ be a semi-cubic bipartite graph with bipartition $(X,Y)$. If $|X|\geq 3$ and $|Y|\geq 2|X|-4$, then $G$ has a subgraph $H$ such that
    \begin{enumerate}
        \item[(\romannum{1})] $H$ is a $K_{3,2}$-subgraph, or
        \item[(\romannum{2})] $H$ is $3$-connected and there is a vertex $h\in V(H)\cap Y$ such that $H-h$ is internally $3$-connected.
    \end{enumerate}
\end{lemma}

\begin{proof}
     We argue by induction on $|X|$. When $|X| = 3$ and $|Y| \geq 2$, it is straightforward to verify that $G$ must contain a $K_{3,2}$-subgraph. When $|X| = 4$ and $|Y| \geq 4$, if $G$ does not contain a $K_{3,2}$-subgraph, then $G$ is isomorphic to the graph shown in Figure~\ref{fig_biregular graph}(a). It is easy to check that $G$ is $3$-connected and, for every $y \in Y$, the graph $G - y$ is isomorphic to the graph in Figure~\ref{fig_biregular graph}(b), which is internally $3$-connected. Therefore, the base cases for $|X| = 3$ and $|X| = 4$ are established.

\begin{center}
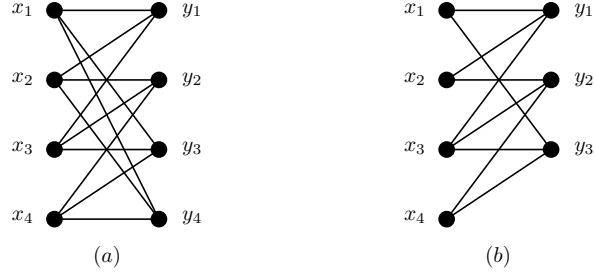
\begin{figure}[htb]
    \hbox to \hsize{
	\hfil
	\resizebox{8cm}{!}{\tikzset{every picture/.style={line width=0.75pt}} 

\begin{tikzpicture}[x=0.75pt,y=0.75pt,yscale=-1,xscale=1]

\draw  [fill={rgb, 255:red, 0; green, 0; blue, 0 }  ,fill opacity=1 ] (132.82,68.2) .. controls (132.82,65.62) and (134.92,63.52) .. (137.5,63.52) .. controls (140.08,63.52) and (142.18,65.62) .. (142.18,68.2) .. controls (142.18,70.78) and (140.08,72.88) .. (137.5,72.88) .. controls (134.92,72.88) and (132.82,70.78) .. (132.82,68.2) -- cycle ;
\draw  [fill={rgb, 255:red, 0; green, 0; blue, 0 }  ,fill opacity=1 ] (132.82,154.86) .. controls (132.82,152.28) and (134.92,150.18) .. (137.5,150.18) .. controls (140.08,150.18) and (142.18,152.28) .. (142.18,154.86) .. controls (142.18,157.45) and (140.08,159.54) .. (137.5,159.54) .. controls (134.92,159.54) and (132.82,157.45) .. (132.82,154.86) -- cycle ;
\draw  [fill={rgb, 255:red, 0; green, 0; blue, 0 }  ,fill opacity=1 ] (132.82,111.53) .. controls (132.82,108.95) and (134.92,106.85) .. (137.5,106.85) .. controls (140.08,106.85) and (142.18,108.95) .. (142.18,111.53) .. controls (142.18,114.12) and (140.08,116.21) .. (137.5,116.21) .. controls (134.92,116.21) and (132.82,114.12) .. (132.82,111.53) -- cycle ;
\draw  [fill={rgb, 255:red, 0; green, 0; blue, 0 }  ,fill opacity=1 ] (132.82,198.2) .. controls (132.82,195.62) and (134.92,193.52) .. (137.5,193.52) .. controls (140.08,193.52) and (142.18,195.62) .. (142.18,198.2) .. controls (142.18,200.78) and (140.08,202.88) .. (137.5,202.88) .. controls (134.92,202.88) and (132.82,200.78) .. (132.82,198.2) -- cycle ;

\draw  [fill={rgb, 255:red, 0; green, 0; blue, 0 }  ,fill opacity=1 ] (197.82,68.2) .. controls (197.82,65.62) and (199.92,63.52) .. (202.5,63.52) .. controls (205.08,63.52) and (207.18,65.62) .. (207.18,68.2) .. controls (207.18,70.78) and (205.08,72.88) .. (202.5,72.88) .. controls (199.92,72.88) and (197.82,70.78) .. (197.82,68.2) -- cycle ;
\draw  [fill={rgb, 255:red, 0; green, 0; blue, 0 }  ,fill opacity=1 ] (197.82,154.86) .. controls (197.82,152.28) and (199.92,150.18) .. (202.5,150.18) .. controls (205.08,150.18) and (207.18,152.28) .. (207.18,154.86) .. controls (207.18,157.45) and (205.08,159.54) .. (202.5,159.54) .. controls (199.92,159.54) and (197.82,157.45) .. (197.82,154.86) -- cycle ;
\draw  [fill={rgb, 255:red, 0; green, 0; blue, 0 }  ,fill opacity=1 ] (197.82,111.53) .. controls (197.82,108.95) and (199.92,106.85) .. (202.5,106.85) .. controls (205.08,106.85) and (207.18,108.95) .. (207.18,111.53) .. controls (207.18,114.12) and (205.08,116.21) .. (202.5,116.21) .. controls (199.92,116.21) and (197.82,114.12) .. (197.82,111.53) -- cycle ;
\draw  [fill={rgb, 255:red, 0; green, 0; blue, 0 }  ,fill opacity=1 ] (197.82,198.2) .. controls (197.82,195.62) and (199.92,193.52) .. (202.5,193.52) .. controls (205.08,193.52) and (207.18,195.62) .. (207.18,198.2) .. controls (207.18,200.78) and (205.08,202.88) .. (202.5,202.88) .. controls (199.92,202.88) and (197.82,200.78) .. (197.82,198.2) -- cycle ;

\draw    (137.5,68.2) -- (202.5,68.2) ;
\draw    (137.5,111.53) -- (202.5,68.2) ;
\draw    (137.5,154.86) -- (202.5,68.2) ;
\draw    (137.5,111.53) -- (202.5,111.53) ;
\draw    (137.5,154.86) -- (202.5,111.53) ;
\draw    (137.5,198.2) -- (202.5,111.53) ;
\draw    (137.5,154.87) -- (202.5,154.87) ;
\draw    (137.5,198.2) -- (202.5,154.86) ;
\draw    (137.5,68.2) -- (202.5,154.86) ;
\draw    (137.5,198.2) -- (202.5,198.2) ;
\draw    (137.5,68.2) -- (202.5,198.2) ;
\draw    (137.5,111.53) -- (202.5,198.2) ;

\draw  [fill={rgb, 255:red, 0; green, 0; blue, 0 }  ,fill opacity=1 ] (376.82,68.2) .. controls (376.82,65.62) and (378.92,63.52) .. (381.5,63.52) .. controls (384.08,63.52) and (386.18,65.62) .. (386.18,68.2) .. controls (386.18,70.78) and (384.08,72.88) .. (381.5,72.88) .. controls (378.92,72.88) and (376.82,70.78) .. (376.82,68.2) -- cycle ;
\draw  [fill={rgb, 255:red, 0; green, 0; blue, 0 }  ,fill opacity=1 ] (376.82,154.86) .. controls (376.82,152.28) and (378.92,150.18) .. (381.5,150.18) .. controls (384.08,150.18) and (386.18,152.28) .. (386.18,154.86) .. controls (386.18,157.45) and (384.08,159.54) .. (381.5,159.54) .. controls (378.92,159.54) and (376.82,157.45) .. (376.82,154.86) -- cycle ;
\draw  [fill={rgb, 255:red, 0; green, 0; blue, 0 }  ,fill opacity=1 ] (376.82,111.53) .. controls (376.82,108.95) and (378.92,106.85) .. (381.5,106.85) .. controls (384.08,106.85) and (386.18,108.95) .. (386.18,111.53) .. controls (386.18,114.12) and (384.08,116.21) .. (381.5,116.21) .. controls (378.92,116.21) and (376.82,114.12) .. (376.82,111.53) -- cycle ;
\draw  [fill={rgb, 255:red, 0; green, 0; blue, 0 }  ,fill opacity=1 ] (376.82,198.2) .. controls (376.82,195.62) and (378.92,193.52) .. (381.5,193.52) .. controls (384.08,193.52) and (386.18,195.62) .. (386.18,198.2) .. controls (386.18,200.78) and (384.08,202.88) .. (381.5,202.88) .. controls (378.92,202.88) and (376.82,200.78) .. (376.82,198.2) -- cycle ;

\draw  [fill={rgb, 255:red, 0; green, 0; blue, 0 }  ,fill opacity=1 ] (441.82,68.2) .. controls (441.82,65.62) and (443.92,63.52) .. (446.5,63.52) .. controls (449.08,63.52) and (451.18,65.62) .. (451.18,68.2) .. controls (451.18,70.78) and (449.08,72.88) .. (446.5,72.88) .. controls (443.92,72.88) and (441.82,70.78) .. (441.82,68.2) -- cycle ;
\draw  [fill={rgb, 255:red, 0; green, 0; blue, 0 }  ,fill opacity=1 ] (441.82,154.86) .. controls (441.82,152.28) and (443.92,150.18) .. (446.5,150.18) .. controls (449.08,150.18) and (451.18,152.28) .. (451.18,154.86) .. controls (451.18,157.45) and (449.08,159.54) .. (446.5,159.54) .. controls (443.92,159.54) and (441.82,157.45) .. (441.82,154.86) -- cycle ;
\draw  [fill={rgb, 255:red, 0; green, 0; blue, 0 }  ,fill opacity=1 ] (441.82,111.53) .. controls (441.82,108.95) and (443.92,106.85) .. (446.5,106.85) .. controls (449.08,106.85) and (451.18,108.95) .. (451.18,111.53) .. controls (451.18,114.12) and (449.08,116.21) .. (446.5,116.21) .. controls (443.92,116.21) and (441.82,114.12) .. (441.82,111.53) -- cycle ;
\draw    (381.5,68.2) -- (446.5,68.2) ;
\draw    (381.5,111.53) -- (446.5,68.2) ;
\draw    (381.5,154.86) -- (446.5,68.2) ;
\draw    (381.5,111.53) -- (446.5,111.53) ;
\draw    (381.5,154.86) -- (446.5,111.53) ;
\draw    (381.5,198.2) -- (446.5,111.53) ;
\draw    (381.5,154.87) -- (446.5,154.87) ;
\draw    (381.5,198.2) -- (446.5,154.86) ;
\draw    (381.5,68.2) -- (446.5,154.86) ;

\draw (215.5,62.4) node [anchor=north west][inner sep=0.75pt]    {$y_{1}$};
\draw (215.5,105.4) node [anchor=north west][inner sep=0.75pt]    {$y_{2}$};
\draw (215.5,148.4) node [anchor=north west][inner sep=0.75pt]    {$y_{3}$};
\draw (215.5,191.4) node [anchor=north west][inner sep=0.75pt]    {$y_{4}$};
\draw (109.5,62.4) node [anchor=north west][inner sep=0.75pt]    {$x_{1}$};
\draw (109.5,105.4) node [anchor=north west][inner sep=0.75pt]    {$x_{2}$};
\draw (109.5,148.4) node [anchor=north west][inner sep=0.75pt]    {$x_{3}$};
\draw (109.5,191.4) node [anchor=north west][inner sep=0.75pt]    {$x_{4}$};
\draw (353.5,191.4) node [anchor=north west][inner sep=0.75pt]    {$x_{4}$};
\draw (353.5,148.4) node [anchor=north west][inner sep=0.75pt]    {$x_{3}$};
\draw (353.5,105.4) node [anchor=north west][inner sep=0.75pt]    {$x_{2}$};
\draw (353.5,62.4) node [anchor=north west][inner sep=0.75pt]    {$x_{1}$};
\draw (459.5,62.4) node [anchor=north west][inner sep=0.75pt]    {$y_{1}$};
\draw (459.5,105.4) node [anchor=north west][inner sep=0.75pt]    {$y_{2}$};
\draw (459.5,148.4) node [anchor=north west][inner sep=0.75pt]    {$y_{3}$};
\draw (160,213.4) node [anchor=north west][inner sep=0.75pt]    {$( a)$};
\draw (404,213.4) node [anchor=north west][inner sep=0.75pt]    {$( b)$};

\end{tikzpicture}}%
	\hfil
    }
    \caption{The graph in (a) is $3$-connected, and the graph in (b) is internally $3$-connected.}\label{fig_biregular graph}
\end{figure}
\end{center}
     
     For the induction step, suppose $|X|\geq 5$ and that the statement holds for every bipartite graph $(X',Y')$ with $|X'|<|X|$. Suppose that $G$ does not have a subgraph satisfying (\romannum{1}) or (\romannum{2}). We first establish the following.

\begin{sublemma}\label{sublem:degree>=3}
If $X'$ is a subset of $X$ such that $|X-X'|\geq 3$, then $|N(X')|\geq 2|X'|+1$. In particular, $d_G(x)\geq 3$ for all $x\in X$.
\end{sublemma}

    Suppose that there is such a subset $X'$ of $X$ with $|N(X')|\leq 2|X'|$. Then the graph $G-N[X']$ is a semi-cubic bipartite graph with bipartition $(X-X',Y-N(X'))$, where
    \[|X-X'|\geq 3\quad  \text{and}\quad |Y-N(X')|\geq 2|X-X'|-4.\]
    By the inductive hypothesis, $G-N[X']$ contains a subgraph satisfying (\romannum{1}) or (\romannum{2}). Since this subgraph is also a subgraph of $G$, we obtain a contradiction. Moreover, when $X'=\{x\}$ for some $x\in X$, we deduce that $|N(x)|\geq 3$. Therefore,~\ref{sublem:degree>=3} holds.\\

Next we show that

\begin{sublemma}\label{sublem:G is 3-con}
    $G$ is $3$-connected.
\end{sublemma}

Suppose that $G$ is not $3$-connected and $\{A,B\}$ is a $2$-separation of $G$ such that $V(A)\cap V(B)\cap Y=S$. Let $X_A = X \cap V(A)$ and $X_B = X \cap V(B)$. Define $Y_A=\{y\in Y-S:N_G(y)\subseteq X_A\}$ and $Y_B=\{y\in Y-S:N_G(y)\subseteq X_B\}$. By Lemma~\ref{lem:bipartite general lemma}, we may assume that $A-S$, which equals $G[X_A\cup Y_A]$, is a bipartite graph with bipartition $(X_A,Y_A)$ such that $|X_A|<|X|$ and $ |Y_A| \ge 2|X_A| -4$. Moreover, $d_G(y)=d_{A-S}(y)$ for all $y\in Y_A$. Let $v$ be a vertex in $V(A)-V(B)$. If $v\in Y$, then, since $G$ is semi-cubic, $d_G(v)=3$. Otherwise, if $v\in X$, by~\ref{sublem:degree>=3}, $d_G(v)\geq 3$. Since $d_G(v)\geq 3$ in each case, $V(A)-V(B)$ contains a neighbor of $v$ and hence $V(A)-V(B)$ contains a pair of adjacent vertices $(x,y)\in X\times Y$. Because $N_G(y)\subseteq X_A$, it is clear that $|X_A|\geq 3$. Thus, by the inductive hypothesis, $A-S$ contains a subgraph satisfying (\romannum{1}) or (\romannum{2}). Hence so does $G$, a contradiction. Therefore,~\ref{sublem:G is 3-con} holds.\\

We may now assume that, for all vertices $y \in Y$, the graph $G - y$ is not $3$-connected, otherwise, $G$ itself would be a subgraph satisfying (\romannum{2}). Let $u$ be an arbitrary vertex in $Y$. Because $|V(G-u)|\geq 4$, we know there is a $2$-separation $\{C,D\}$ of $G-u$. Next we show that

\begin{sublemma}\label{sublem: less than three vertices in X}
    $\min\{|V(C)\cap X|,|V(D)\cap X|\}<3$.
\end{sublemma}

Suppose that $\min\{|V(C)\cap X|,|V(D)\cap X|\}\geq3$. It is clear that $G-u$ is a semi-cubic bipartite graph with bipartition $(X,Y-\{u\})$, so $|Y-\{u\}|\geq 2|X|-5$. Let $T=(Y-\{u\})\cap V(C)\cap V(D)$. Let $X_C = X \cap V(C)$ and $X_D = X \cap V(D)$. Define $Y_C=\{y\in Y-\{u\}-T:N_G(y)\subseteq X_C\}$ and $Y_D=\{y\in Y-\{u\}-T:N_G(y)\subseteq X_D\}$. By Lemma~\ref{lem:bipartite general lemma}, we may assume that $C-T$ is a bipartite graph $G[X_C\cup Y_C]$ with bipartition $(X_C,Y_C)$ such that $|X_C|<|X|$ and $ |Y_C| \ge 2|X_C| -4$. Because $|X_C|\geq 3$, by the inductive hypothesis, we know that $C-T$ contains a subgraph satisfying the conditions described in (\romannum{1}) or (\romannum{2}). Hence so does $G$, a contradiction. Therefore,~\ref{sublem: less than three vertices in X} holds.\\

Without loss of generality, we may assume that $|V(C)\cap X|<3$. Next we show that

\begin{sublemma}\label{sublem:iso to P3}
    $C$ is isomorphic to $P_3$.
\end{sublemma}

We may assume that $(V(C)-V(D))\cap Y=\emptyset$; otherwise, if $p\in (V(C)-V(D))\cap Y$, then $N(p)\subseteq V(C)$ and hence $|V(C)\cap X|\geq 3$, a contradiction. Therefore, $V(C)-V(D)\subseteq X$. For all $x\in V(C)-V(D)$, by~\ref{sublem:degree>=3}, $d_G(x)\geq 3$, so $N(x)=\{c,d,u\}$. If $x_1,x_2$ are two distinct vertices in $V(C)-V(D)$, then $|X-\{x_1,x_2\}|\geq 3$ and $|N(\{x_1,x_2\})|=3 < 2|\{x_1,x_2\}|+1$, contradicting~\ref{sublem:degree>=3}. Thus $|V(C)-V(D)|<2$. Moreover, since $\{C,D\}$ is a $2$-separation, we know $V(C)-V(D)\neq \emptyset$. Hence $V(C)-V(D)=\{x\}$ for some $x\in X$. Because $N(x)=\{c,d,u\}$ and $G-u$ is bipartite, $c$ and $d$ are not adjacent. Thus $C$ is isomorphic to $P_3$, so~\ref{sublem:iso to P3} holds.\\

By~\ref{sublem:G is 3-con} and~\ref{sublem:iso to P3}, we know that $G$ is $3$-connected and that $G - u$ is internally $3$-connected. Therefore, $G$ itself satisfies~(\romannum{2}), a contradiction. This completes the proof of the lemma.
\end{proof}

\begin{lemma}\label{lem:y>=2x-3}
    Let $G$ be a bipartite graph with bipartition $(X,Y)$ such that $d_G(y)\geq 3$ for all $y\in Y$. If $|X|\geq 3$ and $|Y|\geq 2|X|-3$, then $Y$ has a subset $Z$ such that the induced subgraph $G[N[Z]]$ is $3$-connected.
\end{lemma}

\begin{proof}
    We argue by induction on $|X|$. When $|X|=3$, by assumption, it follows that $G$ is a complete bipartite graph $K_{3,|Y|}$ and $|Y|\geq 3$. Let $Z$ be a $3$-element subset of $Y$. Clearly $G[N[Z]]$ is $K_{3,3}$ and so is $3$-connected. Therefore, the base case holds.

    For the inductive step, suppose that $|X|=n\geq 4$ and that the result holds for $|X|<n$. Suppose that $Y$ does not have a subset $Z$ such that $G[N[Z]]$ is $3$-connected. It is clear that $G$ is not $3$-connected, otherwise $G[N[Y]]$ is $3$-connected, a contradiction. Therefore, $G$ has a $2$-separation $\{A,B\}$ with $V(A)\cap V(B)=\{a,b\}$. First we show the following.

    \begin{sublemma}\label{sublem:degree>=3 (2)}
        $d_G(x)\geq 3$ for all $x\in X$.
    \end{sublemma}

    Suppose there is a vertex $x\in X$ such that $d_G(x)<3$. Let $X'=X-\{x\}$ and $Y'=Y-N(x)$. It is clear that $G[X'\cup Y']$ is a bipartite graph with bipartition $(X',Y')$ such that $d_G(y)\geq 3$ for all $y\in Y'$. Moreover, $|X'|\geq 3$ and $|Y'|\geq |Y|-2\geq 2|X|-5= 2|X'|-3$. By the inductive hypothesis, there is a set $Z\subseteq Y'$ such that $G[N[Z]]$ is $3$-connected, a contradiction. Therefore,~\ref{sublem:degree>=3 (2)} holds.\\

    Next we show that
    \begin{sublemma}\label{sublem: intersect X >=3}
        $\min\{|V(A)\cap X|,|V(B)\cap X|\}\geq 3$.
    \end{sublemma}

    Let $w$ be a vertex in $V(A)-V(B)$. Because $d_A(w)=d_G(w)\geq 3$, the vertex $w$  has a neighbor $z$ in $V(A)-\{a,b\}$. Evidently, one of $w$ and $z$ lies in $Y$, and all the neighbors of that vertex are contained in $A$. Therefore $|V(A)\cap X|\geq 3$ and, by symmetry, $|V(B)\cap X|\geq 3$. Thus~\ref{sublem: intersect X >=3} holds.\\

    Let $X_A = X \cap V(A)$ and $X_B = X \cap V(B)$. Define $Y_A=\{y\in Y-\{a,b\}:N_G(y)\subseteq X_A\}$ and $Y_B=\{y\in Y-\{a,b\}:N_G(y)\subseteq X_B\}$. By Lemma~\ref{lem:bipartite general lemma}, we may assume that $A-(Y\cap\{a,b\})$ is a bipartite graph $G[X_A\cup Y_A]$ with bipartition $(X_A,Y_A)$ such that $|X_A|<|X|$ and $|Y_A|\geq 2|X_A|-3$. By~\ref{sublem: intersect X >=3}, $X_A$ has at least three vertices. By the inductive hypothesis, there is a set $Z$ contained in $Y_A$ such that $G[N[Z]]$ is $3$-connected, a contradiction.
\end{proof}

The following is an immediate consequence of the last lemma.

\begin{corollary}\label{cor:y>= 2x-2}
    Let $G$ be a bipartite graph with bipartition $(X,Y)$ such that $d_G(y)\geq 3$ for all $y\in Y$. If $|X|\geq 3$ and $|Y|\geq 2|X|-2$, then $Y$ has a proper subset $Z$ such that the induced subgraph $G[N[Z]]$ is $3$-connected.
\end{corollary}

\section{The minimum number of degree-$3$ vertices}
Before presenting the proof of Theorem~\ref{thm:main}, we prove the following lemma. Let $G$ be a graph and let $H$ be a subgraph of $G$. A vertex $v$ of $G$ is {\it $H$-private} if $N[v]\subseteq V(H)$.

\begin{lemma}\label{lem:H-private}
    Let $G$ be a $3$-connected graph, and let $H$ be a $3$-connected subgraph of $G$. Suppose $v$ is $H$-private, and $\{A,B\}$ is a $3$-separation of $G$ such that $v \in V(A)\cap V(B)$. Then $\{(A\cap H)-v,(B\cap H)-v\}$ is a $2$-separation of $H-v$.
\end{lemma}

\begin{proof}
Let $V(A)\cap V(B)=\{v,a,b\}$. First, we show the following.
\begin{sublemma}\label{sublem:not contained in AB}
    $V(H)\not\subseteq V(A)$ and $V(H)\not\subseteq V(B)$.
\end{sublemma}

Suppose that $V(H)\subseteq V(A)$. Since $v$ is $H$-private, there is no vertex in $V(B)-V(A)$ that is adjacent to $v$. Therefore, $\{a,b\}$ is a vertex-cut of $G$, contradicting the fact that $G$ is $3$-connected. Thus $V(H)\not\subseteq V(A)$ and, by symmetry, ~\ref{sublem:not contained in AB} holds.

\begin{sublemma}\label{sublem:is 3-sep}
    $\{A\cap H,B\cap H\}$ is a $3$-separation of $H$.
\end{sublemma}

By~\ref{sublem:not contained in AB}, neither $(A\cap H)-(B\cap H)$ nor $(B\cap H)-(A\cap H)$ is empty. Therefore, $\{A\cap H,B\cap H\}$ is an $|A\cap B\cap H|$-separation of $H$. Since $H$ is $3$-connected, we have $3\leq|A\cap B\cap H|$. Because $A\cap B=\{v,a,b\}$, we see that $A\cap B\cap H=\{v,a,b\}$. Thus~\ref{sublem:is 3-sep} holds and the lemma follows immediately.
\end{proof}

Let $G$ be a graph with $v(G)$ vertices. Let $V_3(G)$ denote the set of degree-$3$ vertices of $G$, and let $V_+(G)$ denote the set of vertices of $G$ with degree at least four. We write $v_3(G)$ for $|V_3(G)|$ and $v_+(G)$ for $|V_+(G)|$. We now prove the following strengthening of Theorem~\ref{thm:main}.

\begin{theorem}
    If $G$ is both minimally $3$-connected and critically $3$-connected, then
    \[v_3(G)\geq \frac{v(G)+3}{2}.\]
\end{theorem}

\begin{proof}

Let $F$ be the subgraph of $G$ induced by $V_+(G)$, and let $G' = G / E(F)$. For each $i \in \{0,1,2,3\}$, define 
\[ T_i = \left\{ v \in V_3(G) : |N(v) \cap V_+(G)| = i \right\}, \quad \text{and let} \quad t_i = |T_i|. \]
Evidently, we have
\[ v_3(G) = t_0 + t_1 + t_2 + t_3.\tag{1}\label{eq:one} \]
Moreover, each vertex in $T_3$ has all of its neighbors in different components of $F$. Counting the number of edges between $V_+(G)$ and $V_3(G)$ in two different ways, we have

\[
\sum_{v \in V(F)} d_G(v) - 2|E(F)| = \sum_{i \in \{0,1,2,3\}} i \cdot t_i. \tag{2}\label{eq:two}
\]
By Lemma~\ref{lem:Mader_induce_forest}, $F$ is a forest. We note that $|V(F)|=v_+(G)$ and the number of components of $F$ is $v_+(G')$. Therefore, $|E(F)|=v_+(G)-v_+(G')$. Because $d_G(v)\geq 4$ for all $v\in V(F)$, we get the following inequality from~(\ref{eq:two}):
\[
2v_+(G)+2v_+(G') \leq \sum_{i \in \{0,1,2,3\}} i \cdot t_i. \tag{3}\label{eq:three}
\]

Now we focus on the structure of $G'$. By Lemma~\ref{lem:m3c and c3c}, $G'$ is both minimally and critically $3$-connected. Next we show the following.

\begin{sublemma}\label{sublem:t_3<=2n-5}
    If $v_+(G')\geq 3$, then $t_3\leq 2v_+(G')-5$.
\end{sublemma}

Suppose that $t_3\geq 2v_+(G')-4$. Let $G''=G'[V_+(G')\cup T_3]$. It is straightforward to see that $G''$ is a semi-cubic bipartite graph with bipartition $(V_+(G'),T_3)$. By Lemma~\ref{lem:2n-4}, we can break the remainder of the proof of~\ref{sublem:t_3<=2n-5} into the following two cases.
\begin{enumerate}
    \item[(\rom{1})] $G''$ has a $K_{3,2}$-subgraph $J$, and
    \item[(\rom{2})] $G''$ has a $3$-connected subgraph $H$ having a vertex $h\in V(H)\cap T_3$ such that $H-h$ is internally $3$-connected. 
\end{enumerate}

Assume that (\rom{1}) holds. Let $V(J)\cap T_3=\{a_1,a_2\}$ and $V(J)-\{a_1,a_2\}=\{b_1,b_2,b_3\}$. Let $\overline{J}= G'-\{a_1,a_2\}$. Observe that $\{J,\overline{J}\}$ is a $3$-separation of $G'$. Moreover,

\begin{sublemma}\label{sublem:barJ is 2-con}
    $\overline{J}$ is $2$-connected.
\end{sublemma}

To see this, suppose that $\{A,B\}$ is a $1$-separation of $\overline{J}$ such that $V(A)\cap V(B)=\{w\}$. If $\{b_1,b_2,b_3\}\subseteq V(A)$, then $\{A\cup J,B\}$ is a $1$-separation of the $3$-connected graph $G'$, a contradiction. Similarly, $\{b_1,b_2,b_3\}$ is not contained in $V(B)$. Without loss of generality, we may assume that $b_1\in V(A)-V(B)$, and $\{b_2,b_3\}\subseteq V(B)$. If $|V(A)|\geq 3$, then $\{A,J\cup B\}$ is a $2$-separation of $G'$, a contradiction. Thus $V(A)=\{w,b_1\}$ and $d_A(b_1)\leq 1$. However, because $b_1\in V_+(G')$,
\[4\leq d_{G'}(b_1)=d_J(b_1)+d_{\overline{J}}(b_1)=2+d_{\overline{J}}(b_1)=2+d_{A}(b_1),\]
so $d_{A}(b_1)\geq 2$, a contradiction. Therefore,~\ref{sublem:barJ is 2-con} holds.\\

Because $G'$ is critically $3$-connected, $G'-a_1$ is not $3$-connected and hence has a $2$-separation $\{C,D\}$ such that, up to relabeling, $b_1\in V(C)-V(D)$ and $b_2\in V(D)-V(C)$. However, $G'-a_1$ has at least three internally disjoint $b_1b_2$-paths (see Figure~\ref{fig_K_3,2}), a contradiction. Thus (\rom{1}) does not hold.

\begin{center}
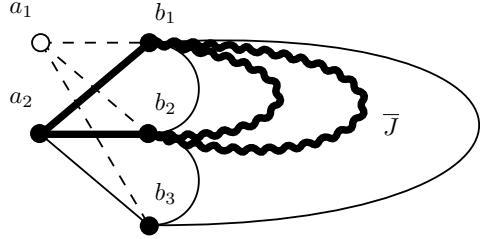
\begin{figure}[htb]
    \hbox to \hsize{
	\hfil
	\resizebox{8cm}{!}{\tikzset{every picture/.style={line width=0.75pt}} 

\begin{tikzpicture}[x=0.75pt,y=0.75pt,yscale=-1,xscale=1]

\draw  [fill={rgb, 255:red, 0; green, 0; blue, 0 }  ,fill opacity=1 ] (207.82,68.6) .. controls (207.82,66.02) and (209.92,63.92) .. (212.5,63.92) .. controls (215.08,63.92) and (217.18,66.02) .. (217.18,68.6) .. controls (217.18,71.18) and (215.08,73.28) .. (212.5,73.28) .. controls (209.92,73.28) and (207.82,71.18) .. (207.82,68.6) -- cycle ;
\draw  [fill={rgb, 255:red, 0; green, 0; blue, 0 }  ,fill opacity=1 ] (207.82,168.2) .. controls (207.82,165.62) and (209.92,163.52) .. (212.5,163.52) .. controls (215.08,163.52) and (217.18,165.62) .. (217.18,168.2) .. controls (217.18,170.78) and (215.08,172.88) .. (212.5,172.88) .. controls (209.92,172.88) and (207.82,170.78) .. (207.82,168.2) -- cycle ;

\draw  [fill={rgb, 255:red, 0; green, 0; blue, 0 }  ,fill opacity=1 ] (207.22,118) .. controls (207.22,115.42) and (209.32,113.32) .. (211.9,113.32) .. controls (214.48,113.32) and (216.58,115.42) .. (216.58,118) .. controls (216.58,120.58) and (214.48,122.68) .. (211.9,122.68) .. controls (209.32,122.68) and (207.22,120.58) .. (207.22,118) -- cycle ;

\draw  [fill={rgb, 255:red, 0; green, 0; blue, 0 }  ,fill opacity=1 ] (148.22,118) .. controls (148.22,115.42) and (150.32,113.32) .. (152.9,113.32) .. controls (155.48,113.32) and (157.58,115.42) .. (157.58,118) .. controls (157.58,120.58) and (155.48,122.68) .. (152.9,122.68) .. controls (150.32,122.68) and (148.22,120.58) .. (148.22,118) -- cycle ;
\draw [fill={rgb, 255:red, 255; green, 255; blue, 255 }  ,fill opacity=1 ] [dash pattern={on 4.5pt off 4.5pt}]  (153.5,68.6) -- (212.5,68.6) ;
\draw [fill={rgb, 255:red, 255; green, 255; blue, 255 }  ,fill opacity=1 ] [dash pattern={on 4.5pt off 4.5pt}]  (153.5,68.6) -- (211.9,118) ;
\draw [fill={rgb, 255:red, 255; green, 255; blue, 255 }  ,fill opacity=1 ] [dash pattern={on 4.5pt off 4.5pt}]  (153.5,68.6) -- (212.5,168.2) ;
\draw [line width=3]    (153.5,118.4) -- (212.5,118.4) ;
\draw [line width=3]    (153.5,118.4) -- (212.5,68.6) ;
\draw    (153.5,118.4) -- (212.5,168.2) ;
\draw    (212.5,68.6) .. controls (249,70) and (248,119) .. (211.9,118) ;
\draw    (212.5,118.4) .. controls (249,119.8) and (248,168.8) .. (211.9,167.8) ;
\draw    (212.5,68.6) .. controls (452,54) and (451,170) .. (211.9,167.8) ;
\draw  [fill={rgb, 255:red, 255; green, 255; blue, 255 }  ,fill opacity=1 ] (148.82,68.6) .. controls (148.82,66.02) and (150.92,63.92) .. (153.5,63.92) .. controls (156.08,63.92) and (158.18,66.02) .. (158.18,68.6) .. controls (158.18,71.18) and (156.08,73.28) .. (153.5,73.28) .. controls (150.92,73.28) and (148.82,71.18) .. (148.82,68.6) -- cycle ;
\draw [line width=3]    (212.5,68.6) .. controls (215.01,67.24) and (216.72,67.46) .. (217.63,69.27) .. controls (219.59,71.23) and (221.23,71.46) .. (222.54,69.97) .. controls (224.91,68.66) and (226.99,68.99) .. (228.76,70.97) .. controls (229.47,72.78) and (230.94,73.05) .. (233.17,71.76) .. controls (235.38,70.49) and (236.78,70.76) .. (237.38,72.58) .. controls (238.81,74.59) and (240.57,74.97) .. (242.66,73.73) .. controls (244.73,72.5) and (246.37,72.9) .. (247.58,74.92) .. controls (248.7,76.94) and (250.22,77.35) .. (252.15,76.16) .. controls (254.08,74.99) and (255.82,75.53) .. (257.38,77.78) .. controls (258.12,79.79) and (259.69,80.35) .. (262.08,79.45) .. controls (263.86,78.36) and (265.25,78.94) .. (266.26,81.18) .. controls (267.04,83.38) and (268.49,84.09) .. (270.62,83.31) .. controls (273.17,82.86) and (274.58,83.7) .. (274.83,85.85) .. controls (275.12,88.14) and (276.47,89.25) .. (278.9,89.2) .. controls (281.33,89.44) and (282.28,90.69) .. (281.76,92.95) .. controls (280.76,94.9) and (281.17,96.51) .. (283,97.8) .. controls (284.33,99.98) and (283.89,101.54) .. (281.69,102.48) .. controls (279.46,102.81) and (278.36,104.16) .. (278.4,106.53) .. controls (278.13,108.92) and (276.76,109.95) .. (274.28,109.63) .. controls (272.17,108.94) and (270.79,109.7) .. (270.16,111.89) .. controls (269.27,114.12) and (267.67,114.8) .. (265.37,113.92) .. controls (263.53,112.81) and (261.96,113.33) .. (260.67,115.48) .. controls (259.66,117.49) and (257.95,117.94) .. (255.53,116.82) .. controls (253.65,115.51) and (252.07,115.83) .. (250.8,117.77) .. controls (249.41,119.68) and (247.74,119.93) .. (245.81,118.52) .. controls (243.9,117.07) and (242.16,117.25) .. (240.57,119.07) .. controls (239.48,120.82) and (237.97,120.92) .. (236.03,119.37) .. controls (234.12,117.77) and (232.56,117.81) .. (231.34,119.5) .. controls (229.41,121.17) and (227.49,121.14) .. (225.56,119.43) .. controls (224.33,117.72) and (222.68,117.63) .. (220.61,119.18) .. controls (219.14,120.75) and (217.46,120.6) .. (215.57,118.75) -- (212.5,118.4) ;
\draw [line width=3]    (212.5,68.6) .. controls (214.39,66.96) and (216.24,66.99) .. (218.05,68.7) .. controls (219.82,70.41) and (221.62,70.45) .. (223.45,68.84) .. controls (225.27,67.23) and (227.02,67.3) .. (228.69,69.03) .. controls (230.32,70.76) and (232.02,70.84) .. (233.79,69.26) .. controls (235.54,67.68) and (237.19,67.77) .. (238.73,69.53) .. controls (240.24,71.3) and (241.84,71.4) .. (243.53,69.85) .. controls (245.22,68.3) and (246.78,68.42) .. (248.19,70.21) .. controls (249.56,72) and (251.06,72.13) .. (252.7,70.6) .. controls (255.79,69.23) and (257.95,69.46) .. (259.19,71.27) .. controls (260.4,73.08) and (261.78,73.25) .. (263.35,71.76) .. controls (264.92,70.28) and (266.25,70.46) .. (267.36,72.29) .. controls (269.69,74.31) and (271.61,74.59) .. (273.12,73.14) .. controls (274.63,71.7) and (276.44,72.01) .. (278.57,74.07) .. controls (279.42,75.92) and (281.14,76.25) .. (283.71,75.06) .. controls (285.16,73.66) and (286.78,74.01) .. (288.56,76.12) .. controls (289.2,77.99) and (290.71,78.36) .. (293.1,77.24) .. controls (294.51,75.89) and (295.93,76.29) .. (297.36,78.42) .. controls (298.64,80.55) and (300.38,81.1) .. (302.58,80.07) .. controls (304.77,79.08) and (306.35,79.66) .. (307.3,81.81) .. controls (308.09,83.94) and (309.5,84.55) .. (311.53,83.63) .. controls (313.57,82.76) and (315.11,83.54) .. (316.14,85.99) .. controls (316.35,88.05) and (317.64,88.86) .. (320.01,88.43) .. controls (322.38,88.14) and (323.62,89.14) .. (323.71,91.45) .. controls (323.69,93.88) and (324.71,95.08) .. (326.76,95.05) .. controls (329.23,96.16) and (329.9,97.72) .. (328.77,99.73) .. controls (327.26,101.4) and (327.2,103.12) .. (328.61,104.87) .. controls (329.62,106.92) and (328.89,108.4) .. (326.42,109.33) .. controls (324.16,109.46) and (323,110.71) .. (322.95,113.06) .. controls (322.84,115.29) and (321.46,116.3) .. (318.8,116.1) .. controls (316.61,115.45) and (315.16,116.24) .. (314.45,118.48) .. controls (313.48,120.75) and (312.1,121.36) .. (310.3,120.29) .. controls (307.95,119.38) and (306.08,120.03) .. (304.71,122.23) .. controls (303.82,124.22) and (302.47,124.6) .. (300.64,123.36) .. controls (298.13,122.23) and (296.3,122.64) .. (295.16,124.57) .. controls (293.91,126.5) and (292.35,126.77) .. (290.49,125.38) .. controls (288.65,123.95) and (287.01,124.17) .. (285.57,126.02) .. controls (284.02,127.85) and (282.3,128.01) .. (280.41,126.5) .. controls (278.54,124.95) and (276.74,125.04) .. (275.02,126.79) .. controls (274.15,128.49) and (272.75,128.52) .. (270.83,126.89) .. controls (268.94,125.23) and (267.02,125.22) .. (265.07,126.86) .. controls (264.04,128.49) and (262.55,128.44) .. (260.61,126.69) .. controls (258.7,124.92) and (257.18,124.82) .. (256.05,126.41) .. controls (253.83,127.9) and (251.74,127.72) .. (249.79,125.85) .. controls (248.92,124.08) and (247.32,123.88) .. (244.99,125.27) .. controls (243.68,126.78) and (242.04,126.54) .. (240.09,124.57) .. controls (239.26,122.74) and (237.59,122.46) .. (235.09,123.72) .. controls (233.66,125.15) and (231.96,124.83) .. (230.01,122.74) .. controls (229.22,120.87) and (227.5,120.5) .. (224.85,121.61) .. controls (223.3,122.96) and (222.14,122.69) .. (221.36,120.78) .. controls (219.44,118.57) and (217.68,118.11) .. (216.07,119.4) -- (212.5,118.4) ;

\draw (214,45.4) node [anchor=north west][inner sep=0.75pt]    {$b_{1}$};
\draw (214,94.4) node [anchor=north west][inner sep=0.75pt]    {$b_{2}$};
\draw (214,143.4) node [anchor=north west][inner sep=0.75pt]    {$b_{3}$};
\draw (135,45.4) node [anchor=north west][inner sep=0.75pt]    {$a_{1}$};
\draw (135,94.4) node [anchor=north west][inner sep=0.75pt]    {$a_{2}$};
\draw (338,105.4) node [anchor=north west][inner sep=0.75pt]    {$\overline{J}$};

\end{tikzpicture}}%
	\hfil
    }
    \caption{There are at least three internally disjoint $b_1b_2$-paths in $G'-a_1$, as indicated in bold.}\label{fig_K_3,2}
\end{figure}
\end{center}

We now know that (\rom{2}) holds, that is, $G''$ has a $3$-connected subgraph $H$ having a vertex $h\in V(H)\cap T_3$ such that $H-h$ is internally $3$-connected. Because $G'$ is critically $3$-connected, we know $G'$ has a $3$-separation $\{P,Q\}$ such that $h\in V(P)\cap V(Q)$. Let $V(P)\cap V(Q)=\{h,p,q\}$. Let $P'=P\cap H$ and $Q'=Q\cap H$. Since $h$ is $H$-private and $H$ is $3$-connected, by Lemma~\ref{lem:H-private}, we know $\{P'-h,Q'-h\}$ is a $2$-separation of $H-h$. Because $H-h$ is internally $3$-connected, up to relabeling, we may assume that $P'-h$ is isomorphic to $P_3$. Let $z$ be the degree-$2$ vertex in $P'-h$. Clearly, $z$ is adjacent to $h$ in $H$. Since $H$ is a bipartite graph, $P'$ is isomorphic to $K_{1,3}$. Moreover, since $h\in T_3$, we see that $z\in V_+(G')$ and $\{h,p,q\}\subseteq T_3$. Also, $h,p$, and $q$ are $H$-private as $H$ is $3$-connected. Because $d_{G'}(z)\geq 4$, there is a vertex $s$ that is adjacent to $z$ in $G'$ and $s\notin \{h,p,q\}$. Clearly, $s\notin V(H)$. By Menger's Theorem, there are at least three internally disjoint paths from $s$ to three different vertices in $H$. Let $L$ be one of these path joining $s$ with $t\in V(H)$ such that $z\notin V(L)$. Since $h,p$, and $q$ are $H$-private, $t\in V(Q')-V(P')$. Therefore, $zs\cup L$ is a path in $G'-\{h,p,q\}$ that joins $z\in V(P)-V(Q)$ with a vertex $t\in V(Q)-V(P)$. This contradicts the fact that $\{P,Q\}$ is a $3$-separation of $G'$ (see Figure~\ref{fig_second}). Therefore, (\rom{2}) does not hold, and this completes the proof of~\ref{sublem:t_3<=2n-5}.

\begin{center}
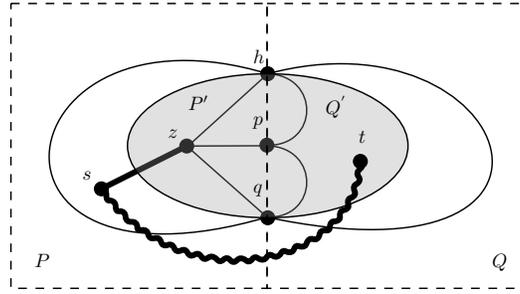
\begin{figure}[htb]
    \hbox to \hsize{
	\hfil
	\resizebox{9cm}{!}{\tikzset{every picture/.style={line width=0.75pt}} 

\begin{tikzpicture}[x=0.75pt,y=0.75pt,yscale=-1,xscale=1]

\draw  [fill={rgb, 255:red, 0; green, 0; blue, 0 }  ,fill opacity=1 ] (227.82,88.6) .. controls (227.82,86.02) and (229.92,83.92) .. (232.5,83.92) .. controls (235.08,83.92) and (237.18,86.02) .. (237.18,88.6) .. controls (237.18,91.18) and (235.08,93.28) .. (232.5,93.28) .. controls (229.92,93.28) and (227.82,91.18) .. (227.82,88.6) -- cycle ;
\draw  [fill={rgb, 255:red, 0; green, 0; blue, 0 }  ,fill opacity=1 ] (227.82,188.2) .. controls (227.82,185.62) and (229.92,183.52) .. (232.5,183.52) .. controls (235.08,183.52) and (237.18,185.62) .. (237.18,188.2) .. controls (237.18,190.78) and (235.08,192.88) .. (232.5,192.88) .. controls (229.92,192.88) and (227.82,190.78) .. (227.82,188.2) -- cycle ;

\draw  [fill={rgb, 255:red, 0; green, 0; blue, 0 }  ,fill opacity=1 ] (227.22,138) .. controls (227.22,135.42) and (229.32,133.32) .. (231.9,133.32) .. controls (234.48,133.32) and (236.58,135.42) .. (236.58,138) .. controls (236.58,140.58) and (234.48,142.68) .. (231.9,142.68) .. controls (229.32,142.68) and (227.22,140.58) .. (227.22,138) -- cycle ;

\draw    (232.5,88.6) .. controls (269,90) and (268,139) .. (231.9,138) ;
\draw    (232.5,138.4) .. controls (269,139.8) and (268,188.8) .. (231.9,187.8) ;
\draw  [fill={rgb, 255:red, 0; green, 0; blue, 0 }  ,fill opacity=1 ] (171.82,138.6) .. controls (171.82,136.02) and (173.92,133.92) .. (176.5,133.92) .. controls (179.08,133.92) and (181.18,136.02) .. (181.18,138.6) .. controls (181.18,141.18) and (179.08,143.28) .. (176.5,143.28) .. controls (173.92,143.28) and (171.82,141.18) .. (171.82,138.6) -- cycle ;
\draw  [fill={rgb, 255:red, 0; green, 0; blue, 0 }  ,fill opacity=1 ] (112.82,168.2) .. controls (112.82,165.62) and (114.92,163.52) .. (117.5,163.52) .. controls (120.08,163.52) and (122.18,165.62) .. (122.18,168.2) .. controls (122.18,170.78) and (120.08,172.88) .. (117.5,172.88) .. controls (114.92,172.88) and (112.82,170.78) .. (112.82,168.2) -- cycle ;
\draw    (232.5,88.6) .. controls (403,43) and (473,239) .. (231.9,187.8) ;
\draw    (232.5,88.6) .. controls (48,33) and (15,251) .. (231.9,187.8) ;
\draw    (177.82,138.2) -- (232.5,88.6) ;
\draw    (176.5,138.6) -- (232.5,138.4) ;
\draw    (176.5,138.6) -- (232.5,188.2) ;
\draw [line width=3]    (117.5,168.2) -- (176.5,138.6) ;
\draw  [fill={rgb, 255:red, 155; green, 155; blue, 155 }  ,fill opacity=0.3 ] (135.57,138.4) .. controls (135.57,110.9) and (178.97,88.6) .. (232.5,88.6) .. controls (286.03,88.6) and (329.43,110.9) .. (329.43,138.4) .. controls (329.43,165.9) and (286.03,188.2) .. (232.5,188.2) .. controls (178.97,188.2) and (135.57,165.9) .. (135.57,138.4) -- cycle ;
\draw  [fill={rgb, 255:red, 0; green, 0; blue, 0 }  ,fill opacity=1 ] (291.82,149.2) .. controls (291.82,146.62) and (293.92,144.52) .. (296.5,144.52) .. controls (299.08,144.52) and (301.18,146.62) .. (301.18,149.2) .. controls (301.18,151.78) and (299.08,153.88) .. (296.5,153.88) .. controls (293.92,153.88) and (291.82,151.78) .. (291.82,149.2) -- cycle ;
\draw [line width=3]    (117.5,168.2) .. controls (119.99,168.81) and (121.14,170.38) .. (120.96,172.92) .. controls (120.86,175.47) and (121.76,176.6) .. (123.66,176.29) .. controls (126.14,176.64) and (127.38,178.07) .. (127.37,180.56) .. controls (127.43,183.07) and (128.39,184.09) .. (130.24,183.6) .. controls (132.7,183.71) and (134.01,184.99) .. (134.17,187.44) .. controls (134.4,189.91) and (135.74,191.1) .. (138.21,191.03) .. controls (140.63,190.87) and (142.01,191.99) .. (142.36,194.38) .. controls (142.06,196.25) and (143.12,197.03) .. (145.53,196.74) .. controls (147.9,196.36) and (149.33,197.34) .. (149.84,199.67) .. controls (150.43,202.01) and (151.89,202.91) .. (154.23,202.37) .. controls (156.52,201.75) and (158,202.57) .. (158.69,204.84) .. controls (159.46,207.09) and (160.97,207.84) .. (163.22,207.07) .. controls (165.41,206.24) and (166.94,206.91) .. (167.8,209.09) .. controls (168.74,211.26) and (170.28,211.85) .. (172.43,210.88) .. controls (174.51,209.84) and (176.06,210.36) .. (177.09,212.45) .. controls (178.19,214.52) and (179.75,214.98) .. (181.78,213.81) .. controls (183.75,212.58) and (185.32,212.96) .. (186.5,214.95) .. controls (187.75,216.91) and (189.32,217.22) .. (191.22,215.88) .. controls (193.05,214.49) and (195.02,214.79) .. (197.13,216.76) .. controls (198.54,218.59) and (200.11,218.74) .. (201.84,217.22) .. controls (203.51,215.67) and (205.07,215.76) .. (206.54,217.49) .. controls (208.08,219.2) and (209.64,219.22) .. (211.22,217.56) .. controls (212.73,215.87) and (214.66,215.81) .. (217.01,217.38) .. controls (218.68,218.93) and (220.21,218.81) .. (221.6,217.02) .. controls (222.91,215.19) and (224.42,215.01) .. (226.13,216.47) .. controls (228.66,217.75) and (230.51,217.43) .. (231.7,215.52) .. controls (232.81,213.59) and (234.27,213.27) .. (236.08,214.56) .. controls (237.95,215.8) and (239.73,215.31) .. (241.43,213.1) .. controls (242.32,211.08) and (243.71,210.63) .. (245.61,211.74) .. controls (248.26,212.54) and (249.96,211.89) .. (250.7,209.8) .. controls (251.35,207.71) and (252.66,207.13) .. (254.64,208.06) .. controls (257.32,208.6) and (258.9,207.8) .. (259.39,205.65) .. controls (260.38,203.17) and (261.89,202.28) .. (263.94,202.99) .. controls (266.03,203.62) and (267.47,202.65) .. (268.26,200.08) .. controls (268.36,197.93) and (269.72,196.88) .. (272.34,196.93) .. controls (274.48,197.3) and (275.51,196.41) .. (275.42,194.24) .. controls (275.71,191.66) and (276.91,190.47) .. (279.03,190.68) .. controls (281.64,190.27) and (282.75,189) .. (282.35,186.89) .. controls (282.24,184.31) and (283.24,182.97) .. (285.36,182.88) .. controls (287.88,182.09) and (288.78,180.69) .. (288.05,178.66) .. controls (287.54,176.13) and (288.33,174.66) .. (290.41,174.23) .. controls (292.77,173.04) and (293.44,171.5) .. (292.41,169.6) .. controls (291.5,167.17) and (292.04,165.56) .. (294.03,164.78) .. controls (296.18,163.17) and (296.59,161.5) .. (295.27,159.77) .. controls (293.88,158.12) and (294.15,156.39) .. (296.1,154.57) .. controls (297.9,153.29) and (298.02,151.86) .. (296.45,150.29) -- (296.5,149.2) ;
\draw  [dash pattern={on 4.5pt off 4.5pt}] (55,40) -- (232,40) -- (232,237) -- (55,237) -- cycle ;
\draw  [dash pattern={on 4.5pt off 4.5pt}] (232,40) -- (411,40) -- (411,237) -- (232,237) -- cycle ;

\draw (221,70.4) node [anchor=north west][inner sep=0.75pt]    {$h$};
\draw (221,116.9) node [anchor=north west][inner sep=0.75pt]    {$p$};
\draw (221,163.4) node [anchor=north west][inner sep=0.75pt]    {$q$};
\draw (162,125.4) node [anchor=north west][inner sep=0.75pt]    {$z$};
\draw (103,154.4) node [anchor=north west][inner sep=0.75pt]    {$s$};
\draw (294,126.4) node [anchor=north west][inner sep=0.75pt]    {$t$};
\draw (70,211.4) node [anchor=north west][inner sep=0.75pt]    {$P$};
\draw (176,102.4) node [anchor=north west][inner sep=0.75pt]    {$P'$};
\draw (386,211.4) node [anchor=north west][inner sep=0.75pt]    {$Q$};
\draw (271,102.4) node [anchor=north west][inner sep=0.75pt]    {$Q^{'}$};

\end{tikzpicture}}%
	\hfil
    }
    \caption{The path in bold cannot exist if $\{P,Q\}$ is a $3$-separation of $G'$.}\label{fig_second}
\end{figure}
\end{center}

Now suppose that $v_+(G)\geq 3$. Then, by~\ref{sublem:t_3<=2n-5}, $2v_+(G')\geq t_3+5$. After substituting into~(\ref{eq:three}) and using~(\ref{eq:one}), we get that
\[2v_+(G)\leq 2t_3+2t_2+t_1-5\leq 2v_3(G)-5.\]
Therefore $2v_3(G)+2v_+(G)\leq 4v_3(G)-5$, so
\[v_3(G)\geq \frac{2v(G)+5}{4}.\]
Since $v_3(G)$ has to be an integer, we have that
\[v_3(G)\geq \frac{v(G)+3}{2}. \tag{4}\label{eq:four}\]

Now we consider the case that $v_+(G')<3$. Clearly, when $v_+(G')=0$,
\[v(G)=v_3(G).\tag{5}\label{eq:five}\]

\begin{sublemma}\label{sublem:if hubs<2}
    If $v_+(G')=1$, then $T_3=T_2=\emptyset$ and $v_3(G)\geq \frac{2v(G)+2}{3}$.
\end{sublemma}
Suppose there is a vertex $t \in T_2\cup T_3$. By Lemma~\ref{lem:m3c and c3c}, $G'$ is simple. Therefore, $|N_{G'}(t) \cap V_+(G')| \geq 2$, contradicting the fact that $v_+(G') = 1$. Thus, $T_3=T_2=\emptyset$. Therefore, from~(\ref{eq:two}), we have that
\[2v_+(G)+2\leq t_1\leq v_3(G).\]
Thus, $2v_+(G)+2v_3(G)+2\leq 3v_3(G)$, so
\[v_3(G)\geq \frac{2v(G)+2}{3}.\tag{6}\label{eq:six}\]
\\

\begin{sublemma}\label{sublem:if hubs<3}
    If $v_+(G')=2$, then $T_3=\emptyset$, and $T_1\cup T_0\neq\emptyset$. Moreover, $v_3(G)\geq \frac{v(G)+3}{2}$.
\end{sublemma}

Suppose there is a vertex $t \in T_3$. As $G'$ is simple, $|N_{G'}(t) \cap V_+(G')| = 3$, contradicting the fact that $v_+(G') = 2$. Thus, $T_3=\emptyset$. Let $V_+(G')=\{a,b\}$. Suppose that $T_1\cup T_0=\emptyset$ and hence $V(G')=T_2\cup\{a,b\}$. Clearly, $d_{G'-\{a,b\}}(v)=1$ for all $v\in T_2$.
As $G'-\{a,b\}$ is connected, $G'-\{a,b\}$ is isomorphic to $K_2$ and, in particular, $|T_2|=2$. However, $T_2$ is a vertex cut of $G'$, a contradiction. Therefore $|T_1\cup T_0|=m$ for some $m\geq 1$, so $m=t_0+t_1$. From~(\ref{eq:three}), we have that
\[2v_+(G)+4\leq 2t_2+t_1\leq 2v_3(G)-m\leq 2v_3(G)-1.\]
Thus,
\[v_3(G)\geq \frac{2v(G)+5}{4}.\]
Moreover, because $v_3(G)$ is an integer, we can conclude that
\[v_3(G)\geq \frac{v(G)+3}{2}.\tag{7}\label{eq:seven}\]
\\

Evidently, a $3$-connected graph has at least four vertices and the only minimally and critically $3$-connected graph with four vertices is $K_4$. Clearly, $v_3(K_4)=4>\frac{v(K_4)+3}{2}$. Moreover, when $v(G)\geq 5$, we have
\[v(G)\geq \frac{2v(G)+2}{3}\geq \frac{v(G)+3}{2}.\]
Thus, from~(\ref{eq:four}),~(\ref{eq:five}),~(\ref{eq:six}), and~(\ref{eq:seven}), we conclude that if $G$ is both minimally and critically $3$-connected, in particular, if $G$ is super-minimally $3$-connected, then
\[
v_3(G) \geq \frac{v(G) + 3}{2}.
\]
\end{proof}

\subsection{Extremal examples}

In this section, we call a graph $G$ an {\it extremal example} if $G$ is super-minimally $3$-connected and $v_3(G)=\lceil\frac{v(G)+3}{2}\rceil$. It is straightforward to verify that $K_4$, the $4$-, and $5$-spoked wheels, and both graphs in Figure~\ref{fig_small_examples} are extremal examples.

\begin{center}
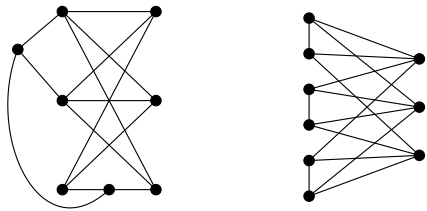
\begin{figure}[htb]
    \hbox to \hsize{
	\hfil
	\resizebox{6cm}{!}{\tikzset{every picture/.style={line width=0.75pt}} 

\begin{tikzpicture}[x=0.75pt,y=0.75pt,yscale=-1,xscale=1]

\draw  [fill={rgb, 255:red, 0; green, 0; blue, 0 }  ,fill opacity=1 ] (366.82,98.88) .. controls (366.82,96.3) and (368.92,94.2) .. (371.5,94.2) .. controls (374.08,94.2) and (376.18,96.3) .. (376.18,98.88) .. controls (376.18,101.47) and (374.08,103.56) .. (371.5,103.56) .. controls (368.92,103.56) and (366.82,101.47) .. (366.82,98.88) -- cycle ;
\draw  [fill={rgb, 255:red, 0; green, 0; blue, 0 }  ,fill opacity=1 ] (366.82,162.88) .. controls (366.82,160.3) and (368.92,158.2) .. (371.5,158.2) .. controls (374.08,158.2) and (376.18,160.3) .. (376.18,162.88) .. controls (376.18,165.47) and (374.08,167.56) .. (371.5,167.56) .. controls (368.92,167.56) and (366.82,165.47) .. (366.82,162.88) -- cycle ;
\draw  [fill={rgb, 255:red, 0; green, 0; blue, 0 }  ,fill opacity=1 ] (366.82,34.88) .. controls (366.82,32.3) and (368.92,30.2) .. (371.5,30.2) .. controls (374.08,30.2) and (376.18,32.3) .. (376.18,34.88) .. controls (376.18,37.47) and (374.08,39.56) .. (371.5,39.56) .. controls (368.92,39.56) and (366.82,37.47) .. (366.82,34.88) -- cycle ;
\draw  [fill={rgb, 255:red, 0; green, 0; blue, 0 }  ,fill opacity=1 ] (366.82,130.88) .. controls (366.82,128.3) and (368.92,126.2) .. (371.5,126.2) .. controls (374.08,126.2) and (376.18,128.3) .. (376.18,130.88) .. controls (376.18,133.47) and (374.08,135.56) .. (371.5,135.56) .. controls (368.92,135.56) and (366.82,133.47) .. (366.82,130.88) -- cycle ;
\draw  [fill={rgb, 255:red, 0; green, 0; blue, 0 }  ,fill opacity=1 ] (366.82,66.88) .. controls (366.82,64.3) and (368.92,62.2) .. (371.5,62.2) .. controls (374.08,62.2) and (376.18,64.3) .. (376.18,66.88) .. controls (376.18,69.47) and (374.08,71.56) .. (371.5,71.56) .. controls (368.92,71.56) and (366.82,69.47) .. (366.82,66.88) -- cycle ;
\draw  [fill={rgb, 255:red, 0; green, 0; blue, 0 }  ,fill opacity=1 ] (366.82,194.88) .. controls (366.82,192.3) and (368.92,190.2) .. (371.5,190.2) .. controls (374.08,190.2) and (376.18,192.3) .. (376.18,194.88) .. controls (376.18,197.47) and (374.08,199.56) .. (371.5,199.56) .. controls (368.92,199.56) and (366.82,197.47) .. (366.82,194.88) -- cycle ;

\draw  [fill={rgb, 255:red, 0; green, 0; blue, 0 }  ,fill opacity=1 ] (465.82,71.55) .. controls (465.82,68.97) and (467.92,66.87) .. (470.5,66.87) .. controls (473.08,66.87) and (475.18,68.97) .. (475.18,71.55) .. controls (475.18,74.14) and (473.08,76.23) .. (470.5,76.23) .. controls (467.92,76.23) and (465.82,74.14) .. (465.82,71.55) -- cycle ;
\draw  [fill={rgb, 255:red, 0; green, 0; blue, 0 }  ,fill opacity=1 ] (465.82,158.21) .. controls (465.82,155.63) and (467.92,153.53) .. (470.5,153.53) .. controls (473.08,153.53) and (475.18,155.63) .. (475.18,158.21) .. controls (475.18,160.8) and (473.08,162.89) .. (470.5,162.89) .. controls (467.92,162.89) and (465.82,160.8) .. (465.82,158.21) -- cycle ;
\draw  [fill={rgb, 255:red, 0; green, 0; blue, 0 }  ,fill opacity=1 ] (465.82,114.88) .. controls (465.82,112.3) and (467.92,110.2) .. (470.5,110.2) .. controls (473.08,110.2) and (475.18,112.3) .. (475.18,114.88) .. controls (475.18,117.47) and (473.08,119.56) .. (470.5,119.56) .. controls (467.92,119.56) and (465.82,117.47) .. (465.82,114.88) -- cycle ;

\draw    (371.5,34.88) -- (470.5,71.55) ;
\draw    (371.5,34.88) -- (371.5,66.88) ;
\draw    (371.5,66.88) -- (470.5,71.55) ;
\draw    (371.5,34.88) -- (470.5,114.88) ;
\draw    (371.5,66.88) -- (470.5,158.21) ;
\draw    (371.5,98.88) -- (470.5,114.88) ;
\draw    (371.5,130.88) -- (470.5,114.88) ;
\draw    (371.5,98.88) -- (470.5,71.55) ;
\draw    (371.5,130.88) -- (470.5,158.21) ;
\draw    (371.5,98.88) -- (371.5,130.88) ;
\draw    (371.5,162.88) -- (470.5,158.21) ;
\draw    (371.5,162.88) -- (371.5,194.88) ;
\draw    (371.5,194.88) -- (470.5,158.21) ;
\draw    (470.5,114.88) -- (371.5,194.88) ;
\draw    (470.5,71.55) -- (371.5,162.88) ;

\draw  [fill={rgb, 255:red, 0; green, 0; blue, 0 }  ,fill opacity=1 ] (145.32,29.02) .. controls (145.32,26.43) and (147.42,24.34) .. (150,24.34) .. controls (152.58,24.34) and (154.68,26.43) .. (154.68,29.02) .. controls (154.68,31.6) and (152.58,33.7) .. (150,33.7) .. controls (147.42,33.7) and (145.32,31.6) .. (145.32,29.02) -- cycle ;
\draw  [fill={rgb, 255:red, 0; green, 0; blue, 0 }  ,fill opacity=1 ] (145.32,109.02) .. controls (145.32,106.43) and (147.42,104.34) .. (150,104.34) .. controls (152.58,104.34) and (154.68,106.43) .. (154.68,109.02) .. controls (154.68,111.6) and (152.58,113.7) .. (150,113.7) .. controls (147.42,113.7) and (145.32,111.6) .. (145.32,109.02) -- cycle ;
\draw  [fill={rgb, 255:red, 0; green, 0; blue, 0 }  ,fill opacity=1 ] (145.32,189.02) .. controls (145.32,186.43) and (147.42,184.34) .. (150,184.34) .. controls (152.58,184.34) and (154.68,186.43) .. (154.68,189.02) .. controls (154.68,191.6) and (152.58,193.7) .. (150,193.7) .. controls (147.42,193.7) and (145.32,191.6) .. (145.32,189.02) -- cycle ;

\draw  [fill={rgb, 255:red, 0; green, 0; blue, 0 }  ,fill opacity=1 ] (229.32,29.02) .. controls (229.32,26.43) and (231.42,24.34) .. (234,24.34) .. controls (236.58,24.34) and (238.68,26.43) .. (238.68,29.02) .. controls (238.68,31.6) and (236.58,33.7) .. (234,33.7) .. controls (231.42,33.7) and (229.32,31.6) .. (229.32,29.02) -- cycle ;
\draw  [fill={rgb, 255:red, 0; green, 0; blue, 0 }  ,fill opacity=1 ] (229.32,109.02) .. controls (229.32,106.43) and (231.42,104.34) .. (234,104.34) .. controls (236.58,104.34) and (238.68,106.43) .. (238.68,109.02) .. controls (238.68,111.6) and (236.58,113.7) .. (234,113.7) .. controls (231.42,113.7) and (229.32,111.6) .. (229.32,109.02) -- cycle ;
\draw  [fill={rgb, 255:red, 0; green, 0; blue, 0 }  ,fill opacity=1 ] (229.32,189.02) .. controls (229.32,186.43) and (231.42,184.34) .. (234,184.34) .. controls (236.58,184.34) and (238.68,186.43) .. (238.68,189.02) .. controls (238.68,191.6) and (236.58,193.7) .. (234,193.7) .. controls (231.42,193.7) and (229.32,191.6) .. (229.32,189.02) -- cycle ;

\draw  [fill={rgb, 255:red, 0; green, 0; blue, 0 }  ,fill opacity=1 ] (187.32,189.02) .. controls (187.32,186.43) and (189.42,184.34) .. (192,184.34) .. controls (194.59,184.34) and (196.68,186.43) .. (196.68,189.02) .. controls (196.68,191.6) and (194.59,193.7) .. (192,193.7) .. controls (189.42,193.7) and (187.32,191.6) .. (187.32,189.02) -- cycle ;
\draw  [fill={rgb, 255:red, 0; green, 0; blue, 0 }  ,fill opacity=1 ] (105.32,63.02) .. controls (105.32,60.43) and (107.42,58.34) .. (110,58.34) .. controls (112.59,58.34) and (114.68,60.43) .. (114.68,63.02) .. controls (114.68,65.6) and (112.59,67.7) .. (110,67.7) .. controls (107.42,67.7) and (105.32,65.6) .. (105.32,63.02) -- cycle ;
\draw    (150,29.02) -- (234,29.02) ;
\draw    (150,109.02) -- (234,109.02) ;
\draw    (150,189.02) -- (234,189.02) ;
\draw    (110,63.02) -- (150,109.02) ;
\draw    (110,63.02) -- (150,29.02) ;
\draw    (150,29.02) -- (234,109.02) ;
\draw    (150,29.02) -- (234,189.02) ;
\draw    (150,109.02) -- (234,29.02) ;
\draw    (150,189.02) -- (234,29.02) ;
\draw    (150,109.02) -- (234,189.02) ;
\draw    (150,189.02) -- (234,109.02) ;
\draw    (110,63.02) .. controls (79.86,140.5) and (129.86,246.5) .. (192,189.02) ;

\end{tikzpicture}}%
	\hfil
    }
    \caption{Two small extremal examples.}\label{fig_small_examples}
\end{figure}
\end{center}

We call the graph in Figure~\ref{fig_links}(a) a {\it link}. A {\it chain of length two} can be constructed as described below (see Figure~\ref{fig_links}(b)). By {\it suppressing a vertex $v$ of degree two in $G$}, we mean deleting $v$ and joining its two neighbors by an edge. 

\begin{enumerate}
    \item[(\romannum{1})] Take two links, $L_1$ and $L_2$.
    \item[(\romannum{2})] Label the right three vertices of $L_1$ from top to bottom as $x$, $y$, and $z$, and label the left three vertices of $L_2$ from top to bottom as $x$, $y$, and $z$.
    \item[(\romannum{3})] Label the remaining vertices of $L_1$ and $L_2$ so that $V(L_1)\cap V(L_2) = \{x, y, z\}$.
    \item[(\romannum{4})] Obtain $I_2$ by suppressing vertices $x$ and $z$ in $L_1\cup L_2$.
\end{enumerate}
Clearly, by using similar steps, we may construct a chain of arbitrary length.

\begin{center}
\begin{figure}[htb]
    \hbox to \hsize{
	\hfil
	\resizebox{10cm}{!}{\input{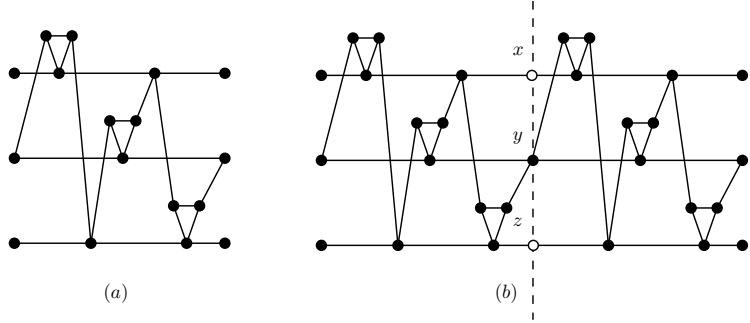}}%
	\hfil
    }
    \caption{A chain of length two is obtained from the graph on the right by suppressing the two white vertices.}\label{fig_links}
\end{figure}
\end{center}

We call the graph in Figure~\ref{fig_belt}(a) a {\it buckle}. A {\it belt of length~$n$}, denoted by $B_n$, is constructed by starting with a chain of length~$n$ and attaching a buckle at each end, as exemplified in Figure~\ref{fig_belt}(b).

\begin{center}
\begin{figure}[htb]
    \hbox to \hsize{
	\hfil
	\resizebox{10cm}{!}{\input{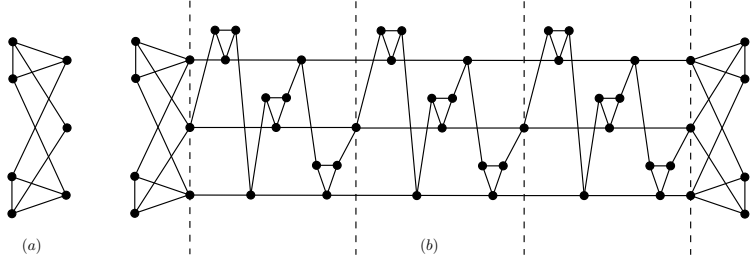}}%
	\hfil
    }
    \caption{(a) A buckle and (b) a belt of length three.}\label{fig_belt}
\end{figure}
\end{center}

For all $n \geq 1$, the graph $B_n$ has $13 + 12n$ vertices, of which $8 + 6n$ have degree three. Moreover, each $B_n$ is super-minimally $3$-connected. To see this, fix $n \geq 1$ and let $G_0 = B_n \setminus e$, where $e$ is an edge contained in one of the buckles. We use $V_-(G)$ to denote the set of vertices in $G$ with degree less than three. Since $e$ is incident to a vertex of degree three, it is clear that $V_-(G_0) \neq \emptyset$. For $i \geq 1$, define the sequence
\[
G_i = G_{i-1} - V_-(G_{i-1}).
\]
If $B_n \setminus e$ has a $3$-connected subgraph $H$, then $H$ must be contained in some $G_j$ for $j \geq 1$, where $V_-(G_j) = \emptyset$ and $G_{j'} = G_j$ for all $j' > j$. To see that $V(G_j)=\emptyset$, consider the subgraph of $G$ in Figure~\ref{fig_25vtx_example}. We may assume that $e$ meets $v_1$ or $v_2$. Then $v_1$ or $v_2$ must be deleted and that necessitates deleting them both. Then the vertices $v_3,v_4,\dots,v_{26}$ are forced to be removed and can be done so in that order. Since we have now deleted all of the vertices of the first link of $G$,  by symmetry, it follows that we are forced to delete all of the vertices of $G$. Therefore, every $3$-connected subgraph of $B_n$ must contain all of the edges in both buckles.

\begin{center}
\begin{figure}[htb]
    \hbox to \hsize{
	\hfil
	\resizebox{8cm}{!}{\input{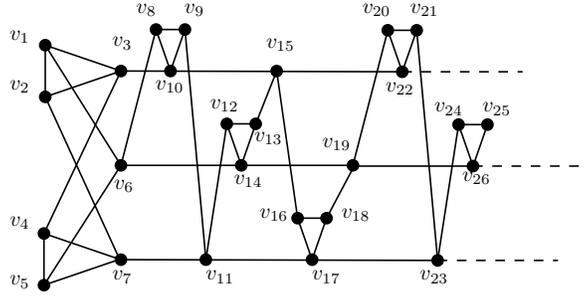}}%
	\hfil
    }
    \caption{Part of a long belt.}\label{fig_25vtx_example}
\end{figure}
\end{center}

If an edge $f$ of order four is deleted from $B_n$, then one can easily check that $B_n \setminus f$ has a $2$-separation $\{X, Y\}$ such that neither $X$ nor $Y$ contains all edges of both buckles. One of the small number of possibilities is shown in Figure~\ref{fig_order_4}. Thus, every $3$-connected subgraph of $B_n$ must contain all edges of order four.

Now let $g$ be an edge that is neither in a buckle nor of order four. Again, checking the small number of cases, one can see that deleting $g$ forces the removal of a degree-$4$ vertex in $B_n$, which in turn leads to the deletion of an edge of order four. Therefore, $g$ is also contained in every $3$-connected subgraph of $B_n$.

We conclude that $B_n$ is super-minimally $3$-connected, and the family of belts forms an infinite collection of extremal examples demonstrating the tightness of the bound in Theorem~\ref{thm:main}.

\begin{center}
\begin{figure}[htb]
    \hbox to \hsize{
	\hfil
	\resizebox{5cm}{!}{\input{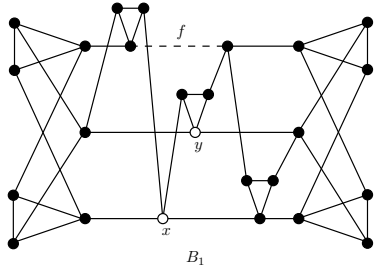}}%
	\hfil
    }
    \caption{$B_1\setminus f$ has a $2$-separation $\{X,Y\}$ such that $V(X)\cap V(Y)=\{x,y\}$.}\label{fig_order_4}
\end{figure}
\end{center}

\section{The maximum number of edges}

In this section, we prove Theorem~\ref{thm:main-density}. Before that we introduce some graph operations that are used in the proof.

\subsection{Bridging}

Let $ab$ be an edge of a graph $G$ and let $x$ be a vertex such that $x\notin\{a,b\}$. We perform a {\it vertex-to-edge bridging on $\{x,ab\}$} by subdividing $ab$ with a new vertex $y$ and then joining $x$ with $y$ (see Figure~\ref{fig_v_to_e}).
\begin{center}
    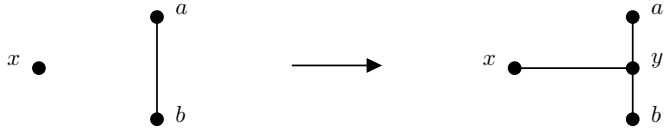
\begin{figure}[htb]
    \hbox to \hsize{
	\hfil
	\resizebox{9cm}{!}{\tikzset{every picture/.style={line width=0.75pt}} 

\begin{tikzpicture}[x=0.75pt,y=0.75pt,yscale=-1,xscale=1]

\draw  [fill={rgb, 255:red, 0; green, 0; blue, 0 }  ,fill opacity=1 ] (159.39,83.59) .. controls (159.39,81.69) and (160.93,80.15) .. (162.83,80.15) .. controls (164.74,80.15) and (166.28,81.69) .. (166.28,83.59) .. controls (166.28,85.49) and (164.74,87.04) .. (162.83,87.04) .. controls (160.93,87.04) and (159.39,85.49) .. (159.39,83.59) -- cycle ;
\draw  [fill={rgb, 255:red, 0; green, 0; blue, 0 }  ,fill opacity=1 ] (159.39,142.5) .. controls (159.39,140.59) and (160.93,139.05) .. (162.83,139.05) .. controls (164.74,139.05) and (166.28,140.59) .. (166.28,142.5) .. controls (166.28,144.4) and (164.74,145.94) .. (162.83,145.94) .. controls (160.93,145.94) and (159.39,144.4) .. (159.39,142.5) -- cycle ;
\draw    (162.83,83.59) -- (162.83,142.5) ;

\draw  [fill={rgb, 255:red, 0; green, 0; blue, 0 }  ,fill opacity=1 ] (91.39,113.04) .. controls (91.39,111.14) and (92.93,109.6) .. (94.83,109.6) .. controls (96.74,109.6) and (98.28,111.14) .. (98.28,113.04) .. controls (98.28,114.95) and (96.74,116.49) .. (94.83,116.49) .. controls (92.93,116.49) and (91.39,114.95) .. (91.39,113.04) -- cycle ;

\draw    (240.5,111.5) -- (289.5,111.5) ;
\draw [shift={(292.5,111.5)}, rotate = 180] [fill={rgb, 255:red, 0; green, 0; blue, 0 }  ][line width=0.08]  [draw opacity=0] (8.93,-4.29) -- (0,0) -- (8.93,4.29) -- cycle    ;
\draw  [fill={rgb, 255:red, 0; green, 0; blue, 0 }  ,fill opacity=1 ] (433.39,83.59) .. controls (433.39,81.69) and (434.93,80.15) .. (436.83,80.15) .. controls (438.74,80.15) and (440.28,81.69) .. (440.28,83.59) .. controls (440.28,85.49) and (438.74,87.04) .. (436.83,87.04) .. controls (434.93,87.04) and (433.39,85.49) .. (433.39,83.59) -- cycle ;
\draw  [fill={rgb, 255:red, 0; green, 0; blue, 0 }  ,fill opacity=1 ] (433.39,142.5) .. controls (433.39,140.59) and (434.93,139.05) .. (436.83,139.05) .. controls (438.74,139.05) and (440.28,140.59) .. (440.28,142.5) .. controls (440.28,144.4) and (438.74,145.94) .. (436.83,145.94) .. controls (434.93,145.94) and (433.39,144.4) .. (433.39,142.5) -- cycle ;
\draw    (436.83,83.59) -- (436.83,142.5) ;

\draw  [fill={rgb, 255:red, 0; green, 0; blue, 0 }  ,fill opacity=1 ] (365.39,113.04) .. controls (365.39,111.14) and (366.93,109.6) .. (368.83,109.6) .. controls (370.74,109.6) and (372.28,111.14) .. (372.28,113.04) .. controls (372.28,114.95) and (370.74,116.49) .. (368.83,116.49) .. controls (366.93,116.49) and (365.39,114.95) .. (365.39,113.04) -- cycle ;

\draw  [fill={rgb, 255:red, 0; green, 0; blue, 0 }  ,fill opacity=1 ] (433.39,113.04) .. controls (433.39,111.14) and (434.93,109.6) .. (436.83,109.6) .. controls (438.74,109.6) and (440.28,111.14) .. (440.28,113.04) .. controls (440.28,114.95) and (438.74,116.49) .. (436.83,116.49) .. controls (434.93,116.49) and (433.39,114.95) .. (433.39,113.04) -- cycle ;
\draw    (368.83,113.04) -- (436.83,113.04) ;

\draw (349,103.9) node [anchor=north west][inner sep=0.75pt]    {$x$};
\draw (446,74.4) node [anchor=north west][inner sep=0.75pt]    {$a$};
\draw (446,133.4) node [anchor=north west][inner sep=0.75pt]    {$b$};
\draw (446,102.9) node [anchor=north west][inner sep=0.75pt]    {$y$};
\draw (75,103.9) node [anchor=north west][inner sep=0.75pt]    {$x$};
\draw (172,74.4) node [anchor=north west][inner sep=0.75pt]    {$a$};
\draw (172,133.4) node [anchor=north west][inner sep=0.75pt]    {$b$};

\end{tikzpicture}}%
	\hfil
    }
    \caption{A vertex-to-edge bridging.}\label{fig_v_to_e}
\end{figure}
\end{center}

Let $ab$ and $cd$ be two distinct edges of $G$. We perform an {\it edge-to-edge bridging on $\{ab,cd\}$} by subdividing $ab$ with $x$, subdividing $cd$ with $y$, and then joining $x$ with $y$ (see Figure~\ref{fig_e_to_e}). We note that $ab$ and $cd$ could be two incident edges, that is, $\{a,b\}\cap \{c,d\}$ is not necessarily empty.

\begin{center}
    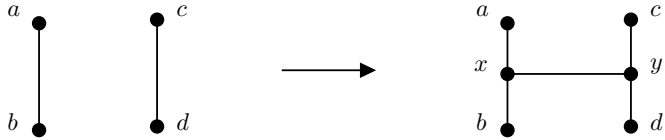
\begin{figure}[htb]
    \hbox to \hsize{
	\hfil
	\resizebox{9cm}{!}{\tikzset{every picture/.style={line width=0.75pt}} 

\begin{tikzpicture}[x=0.75pt,y=0.75pt,yscale=-1,xscale=1]

\draw  [fill={rgb, 255:red, 0; green, 0; blue, 0 }  ,fill opacity=1 ] (123.39,105.59) .. controls (123.39,103.69) and (124.93,102.15) .. (126.83,102.15) .. controls (128.74,102.15) and (130.28,103.69) .. (130.28,105.59) .. controls (130.28,107.49) and (128.74,109.04) .. (126.83,109.04) .. controls (124.93,109.04) and (123.39,107.49) .. (123.39,105.59) -- cycle ;
\draw  [fill={rgb, 255:red, 0; green, 0; blue, 0 }  ,fill opacity=1 ] (123.39,164.5) .. controls (123.39,162.59) and (124.93,161.05) .. (126.83,161.05) .. controls (128.74,161.05) and (130.28,162.59) .. (130.28,164.5) .. controls (130.28,166.4) and (128.74,167.94) .. (126.83,167.94) .. controls (124.93,167.94) and (123.39,166.4) .. (123.39,164.5) -- cycle ;
\draw    (126.83,105.59) -- (126.83,164.5) ;

\draw    (260.5,131.5) -- (309.5,131.5) ;
\draw [shift={(312.5,131.5)}, rotate = 180] [fill={rgb, 255:red, 0; green, 0; blue, 0 }  ][line width=0.08]  [draw opacity=0] (8.93,-4.29) -- (0,0) -- (8.93,4.29) -- cycle    ;
\draw  [fill={rgb, 255:red, 0; green, 0; blue, 0 }  ,fill opacity=1 ] (381.39,133.54) .. controls (381.39,131.64) and (382.93,130.1) .. (384.83,130.1) .. controls (386.74,130.1) and (388.28,131.64) .. (388.28,133.54) .. controls (388.28,135.45) and (386.74,136.99) .. (384.83,136.99) .. controls (382.93,136.99) and (381.39,135.45) .. (381.39,133.54) -- cycle ;
\draw  [fill={rgb, 255:red, 0; green, 0; blue, 0 }  ,fill opacity=1 ] (449.39,133.54) .. controls (449.39,131.64) and (450.93,130.1) .. (452.83,130.1) .. controls (454.74,130.1) and (456.28,131.64) .. (456.28,133.54) .. controls (456.28,135.45) and (454.74,136.99) .. (452.83,136.99) .. controls (450.93,136.99) and (449.39,135.45) .. (449.39,133.54) -- cycle ;
\draw    (384.83,133.54) -- (452.83,133.54) ;

\draw  [fill={rgb, 255:red, 0; green, 0; blue, 0 }  ,fill opacity=1 ] (188.39,103.59) .. controls (188.39,101.69) and (189.93,100.15) .. (191.83,100.15) .. controls (193.74,100.15) and (195.28,101.69) .. (195.28,103.59) .. controls (195.28,105.49) and (193.74,107.04) .. (191.83,107.04) .. controls (189.93,107.04) and (188.39,105.49) .. (188.39,103.59) -- cycle ;
\draw  [fill={rgb, 255:red, 0; green, 0; blue, 0 }  ,fill opacity=1 ] (188.39,162.5) .. controls (188.39,160.59) and (189.93,159.05) .. (191.83,159.05) .. controls (193.74,159.05) and (195.28,160.59) .. (195.28,162.5) .. controls (195.28,164.4) and (193.74,165.94) .. (191.83,165.94) .. controls (189.93,165.94) and (188.39,164.4) .. (188.39,162.5) -- cycle ;
\draw    (191.83,103.59) -- (191.83,162.5) ;

\draw  [fill={rgb, 255:red, 0; green, 0; blue, 0 }  ,fill opacity=1 ] (381.39,105.59) .. controls (381.39,103.69) and (382.93,102.15) .. (384.83,102.15) .. controls (386.74,102.15) and (388.28,103.69) .. (388.28,105.59) .. controls (388.28,107.49) and (386.74,109.04) .. (384.83,109.04) .. controls (382.93,109.04) and (381.39,107.49) .. (381.39,105.59) -- cycle ;
\draw  [fill={rgb, 255:red, 0; green, 0; blue, 0 }  ,fill opacity=1 ] (381.39,164.5) .. controls (381.39,162.59) and (382.93,161.05) .. (384.83,161.05) .. controls (386.74,161.05) and (388.28,162.59) .. (388.28,164.5) .. controls (388.28,166.4) and (386.74,167.94) .. (384.83,167.94) .. controls (382.93,167.94) and (381.39,166.4) .. (381.39,164.5) -- cycle ;
\draw    (384.83,105.59) -- (384.83,164.5) ;

\draw  [fill={rgb, 255:red, 0; green, 0; blue, 0 }  ,fill opacity=1 ] (449.39,103.59) .. controls (449.39,101.69) and (450.93,100.15) .. (452.83,100.15) .. controls (454.74,100.15) and (456.28,101.69) .. (456.28,103.59) .. controls (456.28,105.49) and (454.74,107.04) .. (452.83,107.04) .. controls (450.93,107.04) and (449.39,105.49) .. (449.39,103.59) -- cycle ;
\draw  [fill={rgb, 255:red, 0; green, 0; blue, 0 }  ,fill opacity=1 ] (449.39,162.5) .. controls (449.39,160.59) and (450.93,159.05) .. (452.83,159.05) .. controls (454.74,159.05) and (456.28,160.59) .. (456.28,162.5) .. controls (456.28,164.4) and (454.74,165.94) .. (452.83,165.94) .. controls (450.93,165.94) and (449.39,164.4) .. (449.39,162.5) -- cycle ;
\draw    (452.83,103.59) -- (452.83,162.5) ;

\draw (201,94.4) node [anchor=north west][inner sep=0.75pt]    {$c$};
\draw (201,153.4) node [anchor=north west][inner sep=0.75pt]    {$d$};
\draw (108,94.4) node [anchor=north west][inner sep=0.75pt]    {$a$};
\draw (108,153.4) node [anchor=north west][inner sep=0.75pt]    {$b$};
\draw (462,153.4) node [anchor=north west][inner sep=0.75pt]    {$d$};
\draw (462,94.4) node [anchor=north west][inner sep=0.75pt]    {$c$};
\draw (366,153.4) node [anchor=north west][inner sep=0.75pt]    {$b$};
\draw (366,94.4) node [anchor=north west][inner sep=0.75pt]    {$a$};
\draw (365,124.4) node [anchor=north west][inner sep=0.75pt]    {$x$};
\draw (462,123.4) node [anchor=north west][inner sep=0.75pt]    {$y$};

\end{tikzpicture}}%
	\hfil
    }
    \caption{An edge-to-edge bridging.}\label{fig_e_to_e}
\end{figure}
\end{center}

We omit the elementary proof of the following result.

\begin{proposition}\label{prop:bridging preserve 3-con}
    If $G$ is a $3$-connected graph and $G'$ is obtained from $G$ by applying a vertex-to-edge or an edge-to-edge bridging, then $G'$ is $3$-connected.
\end{proposition}

\begin{lemma}\label{lem:briding a wheel}
    Suppose $G$ is a wheel with at least three spokes and $G'$ is obtained from $G$ by performing a vertex-to-edge bridging on $\{x,ab\}$. Then $G'$ is super-minimally $3$-connected if and only if $x$ is the hub of $G$ and $ab$ is a rim edge. In particular, if $G'$ is super-minimally $3$-connected, then $G'$ is also a wheel.
\end{lemma}

\begin{proof}
    Suppose $G$ is a wheel with $n$ spokes and $G'$ is obtained from $G$ by bridging the hub with a rim edge. Then $G'$ is a wheel with $n+1$ spokes, so $G'$ is super-minimally $3$-connected.

    Conversely, suppose $x$ is the hub of $G$. If $ab$ is not a rim edge, then $G'$ is not simple, a contradiction. Therefore, we may assume that $x$ is not the hub of $G$. This implicitly implies that $G$ has at least four spokes, because, for the three-spoked wheel, every vertex can be viewed as the hub. We label the vertices on the rim of $G$ by $v_1,v_2,\dots,v_n$ consecutively for some $n\geq 4$ and label the hub of $G$ by $h$. Without loss of generality, we may assume that $x=v_1$. Suppose $ab=hv_i$ for some $i\in\{1,2,\dots,n\}$. Then $i\neq 1$, otherwise $G'$ is not simple. For all $i\in\{2,3\dots,n\}$, the graph $G'$ obtained by bridging $v_1$ and $hv_{i}$ has four internally disjoint paths between $v_1$ and $h$, contradicting Lemma~\ref{lem:bollobas}. Thus, $ab$ is a rim edge in $\{v_1v_2,v_2v_3,\dots,v_nv_1\}$. Clearly, $ab\notin\{v_1v_2,v_nv_1\}$, otherwise $G'$ is not simple. However, for all $ab\in \{v_2v_3, v_3v_4\dots,v_{n-1}v_n\}$, the graph $G'$ obtained by bridging $v_1$ and $ab$ has four internally disjoint paths between $v_1$ and $h$, contradicting Lemma~\ref{lem:bollobas}.
\end{proof}

\subsection{Enhanced deletion}

We now consider the reverse of the bridging operations described above. Let $a$ and $b$ be two nonadjacent vertices of a $3$-connected graph $G$, and let $xy$ be an edge of $G$. The graph $G+ab$ is obtained from $G$ by adding the edge $ab$. The \emph{enhanced deletion} operation $G\dd xy$ is defined as follows.

\begin{itemize}
    \item[(\romannum{1})] If, in $G\setminus xy$, we have $N(x)=\{a,b\}$ and $N(y)=\{c,d\}$, where $\{a,b\}$ and $\{c,d\}$ are disjoint pairs of nonadjacent vertices, then $G\dd xy = (G-\{x,y\}) + ab + cd$.

    \item[(\romannum{2})] If, for exactly one $z\in\{x,y\}$, the set $N_{G\setminus xy}(z)$ contains exactly two vertices, say $a$ and $b$, and these vertices are nonadjacent, then $G\dd xy = (G - z) + ab$.

    \item[(\romannum{3})] In all other cases, $G\dd xy = G\setminus xy$.
\end{itemize}

Figure~\ref{fig_G--xy} illustrates the operation $G\dd xy$ in these different cases.  
Our definition of $G\dd xy$ coincides with that of Kriesell~\cite{Kriesell-triangle-free}, but differs from Tutte’s original definition~\cite{Tutte1961}.  
Under our definition, the enhanced deletion $G\dd xy$ is always simple whenever $G$ is simple.

\begin{center}
\begin{figure}[htb]
    \hbox to \hsize{
	\hfil
	\resizebox{6cm}{!}{\input{figures/fig_G--xy}}%
	\hfil
    }
    \caption{$G\dd xy$ in different cases.}\label{fig_G--xy}
\end{figure}
\end{center}

The following result of Kriesell~\cite[Theorem 2]{Kriesell-triangle-free} is used in the proof of Theorem~\ref{thm:main-density}. An edge $e$ in a $3$-connected graph $G$ is {\it $3$-contractible} if the graph $G/e$ is $3$-connected. Since we are only considering simple graphs, after the contraction of $e$, we only keep the underlying simple graph.

\begin{lemma}\label{lem:kriesell}
    Let $G$ be a $3$-connected graph that is not $K_4$. If $xy$ is not $3$-contractible, then $G\dd xy$ is $3$-connected.
\end{lemma}

\subsection{Cleaving}
Clearly, for every edge $ab$ of a super-minimally $3$-connected graph $G$, the graph $G\setminus ab$ is $2$-connected but not $3$-connected.

\begin{lemma}\label{lem:cleavable or internally 3-con}
    Suppose $G$ is a triangle-free super-minimally $3$-connected graph such that $G\setminus ab$ is internally $3$-connected. Then $G\dd ab$ is $3$-connected.
\end{lemma}

\begin{proof}
    A triangle-free $3$-connected graph has at least six vertices, so $|V(G\dd ab)|\geq 4$. Suppose that $G\dd ab$ has a $2$-separation $\{A,B\}$. Because the minimum degree of $G\dd ab$ is at least three, we know that $\min\{|V(A)|,|V(B)|\}\geq 4$. It is straightforward to see that $G\setminus ab$ has a $2$-separation $\{A',B'\}$ such that $A'$ is obtained from $A$ by possibly subdividing an edge, and $B'$ is obtained from $B'$ by possibly subdividing an edge. Therefore, $\min\{|V(A')|,|V(B')|\}\geq 4$ and neither $A'$ nor $B'$ is $P_3$, a contradiction.
\end{proof}

Let $G$ be a $3$-connected graph. We define three types of {\it cleaving} operation as follows.
\begin{enumerate}
    \item[(\romannum{1})] Suppose $\{a_1b_1,a_2b_2,a_3b_3\}$ is a set of edges of $G$ such that $G\setminus\{a_1b_1,a_2b_2,a_3b_3\}$ has a $0$-separation $\{A,B\}$ such that $\{a_1,a_2,a_3\}\subseteq V(A)$ and $\{b_1,b_2,b_3\}\subseteq V(B)$. Then let $G_A$ be a graph obtained from $A$ by adding a vertex $x$ that is adjacent to $a_1,a_2$, and $a_3$, and let $G_B$ be a graph obtained from $B$ by adding a vertex $y$ that is adjacent to $b_1,b_2$, and $b_3$.
    \item[(\romannum{2})]  Suppose $\{a_1b_1,a_2b_2,c\}$ is a set of two edges and one vertex of $G$ such that $G\setminus\{a_1b_1,a_2b_2\}$ has a $1$-separation $\{A,B\}$ for which $\{a_1,a_2,c\}\subseteq V(A)$ and $\{b_1,b_2,c\}\subseteq V(B)$ where $\min\{d_A(c),d_B(c)\}\geq 2$. Then let $G_A$ be a graph obtained from $A$ by adding a vertex $x$ that is adjacent to $a_1,a_2$, and $c$, and let $G_B$ be a graph obtained from $B$ by adding a vertex $y$ that is adjacent to $b_1,b_2$, and $c$.
    \item[(\romannum{3})] Suppose $\{a_1b_1,c,d\}$ is a set of one edge and two vertices of $G$ such that $G\setminus\{a_1b_1,a_2b_2\}$ has a $2$-separation $\{A,B\}$ for which $\{a_1,c,d\}\subseteq V(A)$ and $\{b_1,c,d\}\subseteq V(B)$ where $\min\{d_A(c),d_B(c),d_A(d),d_B(d)\}\geq 2$. Then let $G_A$ be a graph obtained from $A$ by adding a vertex $x$ that is adjacent to $a_1,c$, and $d$, and let $G_B$ be a graph obtained from $B$ by adding a vertex $y$ that is adjacent to $b_1,c$, and $d$.
\end{enumerate}

\begin{center}
\begin{figure}[htb]
    \hbox to \hsize{
	\hfil
	\resizebox{12cm}{!}{\input{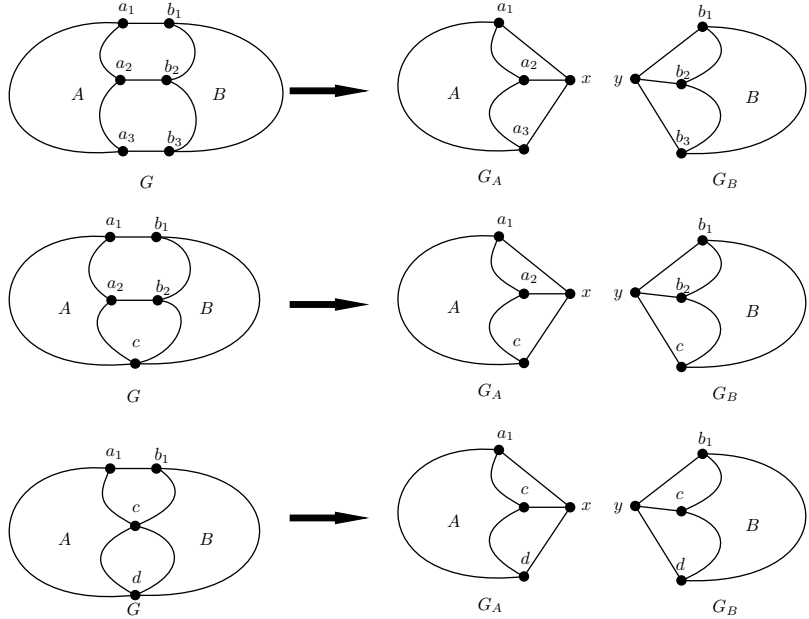}}%
	\hfil
    }
    \caption{Cleaving $G$ into $G_A$ and $G_B$.}\label{fig_cleaving}
\end{figure}
\end{center}

We call the sets of edges and vertices described above {\it compatible sets} of type (\romannum{1}),(\romannum{2}), and (\romannum{3}) respectively, and we shall refer to the process of constructing $G_A$ and $G_B$ as {\it cleaving $G$ with respect to a compatible set}. Figure~\ref{fig_cleaving} illustrates the three different types of cleaving. The following lemma is straightforward, and we omit its proof.

\begin{lemma}\label{lem:exist compatible set}
    Let $G$ be a super-minimally $3$-connected graph. If $G$ has an edge $ab$ such that $G\setminus ab$ is not internally $3$-connected, then $G$ has a compatible set $S$ containing $ab$. 
\end{lemma}

The next lemma shows that the graphs obtained by cleaving super-minimally $3$-connected graphs are also super-minimally $3$-connected.

\begin{lemma}\label{lem:cleaving preserves sm3c}
    If $G$ is a super-minimally $3$-connected graph and $S$ is a compatible set of $G$, then the graphs $G_A$ and $G_B$ obtained by cleaving $G$ with respect to $S$ are both super-minimally $3$-connected.
\end{lemma}

\begin{proof}
The fact that $G_A$ and $G_B$ are $3$-connected follows from Lemma~\ref{lem: cleaving is 3-con}, so we only have to show that they do not have proper $3$-connected subgraphs. Suppose $S=\{a_1b_1,a_2b_2,a_3b_3\}$ and $G_A$, $G_B$ are obtained from a type (\romannum{1}) cleaving. If $G_A'$ is a proper $3$-connected subgraph of $G_A$, then, since $G$ is super-minimally $3$-connected, $G_A'$ contains at least one member of $\{xa_1,xa_2,xa_3\}$. Hence $G_A'$ contains all of $\{xa_1,xa_2,xa_3\}$. Let $G'$ be the graph $(G_A'-x)\cup (G_B-y)+a_1b_1+a_2b_2+a_3b_3$. It is straightforward to check that $G'$ is a proper $3$-connected subgraph of $G$, a contradiction. By symmetry, $G_B$ does not have a proper $3$-connected subgraph. For the other two types of cleaving operations, the proofs are similar and they are omitted.
\end{proof}

\setcounter{theorem}{6}
\begin{proof}[Proof of Theorem~\ref{thm:main-density}]
We argue by induction on $|V(G)|$. Clearly, $K_4$ is the only super-minimally $3$-connected graph that has fewer than five vertices. As $|E(K_4)|=2|V(K_4)|-2$, and $K_4$ is a wheel with three spokes, the base case holds. For the inductive step, suppose $|V(G)|\geq 5$ and the statement holds for all super-minimally $3$-connected graphs that have at most $|V(G)|-1$ vertices. We will show that $|E(G)|\leq 2|V(G)|-3$ unless $G$ is a wheel. Suppose, on the contrary, that $|E(G)|\geq 2|V(G)|-2$ and $G$ is not a wheel. First we show that

 \begin{sublemma}\label{sublem:3-contractible}
     every edge of $G$ is $3$-contractible.
 \end{sublemma}

 Suppose $G/xy$ is not $3$-connected. Clearly $x,y$ are in a $3$-cut $\{x,y,z\}$ of $G$. Moreover, by Lemma~\ref{lem:kriesell}, $G\dd xy$ is $3$-connected. It is straightforward to see that at least one of $x$ and $y$ is a degree-$3$ vertex, otherwise $G\dd xy=G\setminus xy$, a contradiction. Without loss of generality, we may assume that $N_G(x)=\{a,b,y\}$. There are two possible cases as shown in Figure~\ref{fig_3-contractible}. 
 
  \begin{center}
\begin{figure}[htb]
    \hbox to \hsize{
	\hfil
	\resizebox{10cm}{!}{\input{figures/fig_3-contractible}}%
	\hfil
    }
    \caption{Two possible cases of $G\dd xy$.}\label{fig_3-contractible}
\end{figure}
\end{center}
 
 Suppose $d_G(y)\geq 4$. Then $G\dd xy=(G-x)+ab$. If $G\dd xy$ has a proper $3$-connected subgraph $H$, then $H$ must contain $\{ab,y,z\}$, otherwise $H$ will also be a proper $3$-connected subgraph of $G$. Let $H'$ be obtained by applying a vertex-to-edge bridging on $\{y,ab\}$. By Proposition~\ref{prop:bridging preserve 3-con}, $H'$ is $3$-connected. Clearly, $H'$ is a proper $3$-connected subgraph of $G$, a contradiction. Therefore, $G\dd xy$ is super-minimally $3$-connected. Moreover,
 \[|E(G)|=|E(G\dd xy)|+2\leq 2|V(G\dd xy)|=2|V(G)|-2.\]
 Because $|E(G)|\geq 2|V(G)|-2$, we know $|E(G\dd xy)|=2|V(G\dd xy)|-2$ and, by the inductive hypothesis, $G\dd xy$ is a wheel on at least three spokes. By Lemma~\ref{lem:briding a wheel}, $G$ is a wheel, a contradiction.

 We may now assume that $N_G(y)=\{x,c,d\}$. Then $G\dd xy=(G-\{x,y\})+ab+cd$. If $G\dd xy$ has a proper $3$-connected subgraph $H$, then $H$ must contain $\{ab,cd,z\}$, otherwise $H$ will also be a proper $3$-connected subgraph of $G$. Let $H'$ be the obtained by applying a edge-to-edge bridging on $\{ab,cd\}$. By Proposition~\ref{prop:bridging preserve 3-con}, $H'$ is $3$-connected. Clearly, $H'$ is a proper $3$-connected subgraph of $G$, a contradiction. Therefore, $G\dd xy$ is super-minimally $3$-connected. Moreover,
 \[|E(G)|=|E(G\dd xy)|+3\leq 2|V(G\dd xy)|+1=2|V(G)|-3,\]
 a contradiction. Hence~\ref{sublem:3-contractible} holds. \\
 
 We show next that

 \begin{sublemma}\label{sublem:triangle-free}
     $G$ is triangle-free.
 \end{sublemma}

 Suppose $G$ has a triangle on vertices $\{a,b,c\}$. By Lemma~\ref{lem:Halin_two_deg3_vtx}, it contains two vertices of degree three. Suppose $N_G(b)=\{a,c,x\}$ and $N_G(c)=\{a,b,y\}$. It is straightforward to check that $ab$ is not $3$-contractible as $c$ only has two neighbors in $G/ab$, a contradiction of~\ref{sublem:3-contractible}.\\

 Next we show the following.

 \begin{sublemma}\label{sublem:G--xy is 3-con}
     For every edge $xy$ of $G$, the graph $G\dd xy$ is $3$-connected.
 \end{sublemma}

 Suppose that $G\dd xy$ is not $3$-connected. By Lemma~\ref{lem:cleavable or internally 3-con}, $G\setminus xy$ is not internally $3$-connected and, by Lemma~\ref{lem:exist compatible set}, $G$ has a compatible set $S$ containing $xy$. Let $G_A$ and $G_B$ be the graphs obtained by cleaving $G$ with respect to $S$. By Lemma~\ref{lem:cleaving preserves sm3c}, both $G_A$ and $G_B$ are super-minimally $3$-connected. Moreover, since $G$ is triangle-free, it is clear that neither $G_A$ nor $G_B$ is a wheel. Hence, by the inductive hypothesis, $|E(G_A)|\leq 2|V(G_A)|-3$ and $|E(G_B)|\leq 2|V(G_B)|-3$.

If $S$ is a type (\romannum{1}) compatible set, then
 \[|E(G)|=|E(G_A)|+|E(G_B)|-3\leq 2|V(G_A)|+2|V(G_B)|-9=2|V(G)|-5. \]
If $S$ is a type (\romannum{2}) compatible set, then
 \[|E(G)|=|E(G_A)|+|E(G_B)|-4\leq 2|V(G_A)|+2|V(G_B)|-10=2|V(G)|-4. \]
If $S$ is a type (\romannum{3}) compatible set, then
 \[|E(G)|=|E(G_A)|+|E(G_B)|-5\leq 2|V(G_A)|+2|V(G_B)|-11=2|V(G)|-3. \]
In all cases, $|E(G)|\leq 2|V(G)|-3$, a contradiction. Thus~\ref{sublem:G--xy is 3-con} holds.\\

It follows immediately that

\begin{sublemma}\label{sublem:all incident to deg-3}
    every edge in $G$ is incident to a degree-$3$ vertex.
\end{sublemma}

Let $Z$ be the set of vertices of $G$ that have degree at least four, and let $C_1,C_2,\dots,C_k$ be the connected components of $G-Z$. Clearly, $Z\neq\emptyset$, otherwise $G$ is a cubic graph and $|E(G)|=\frac{3|V(G)|}{2}\leq2|E(G)|-2$ with equality holding only if $G=K_4$. Next we show the following.

\begin{sublemma}\label{sublem:one edge}
    For all $i\in\{1,2,\dots,k\}$ and $z\in Z$, there is at most one edge between $V(C_i)$ and $z$. In particular, $|Z|\geq 3$.
\end{sublemma}

Suppose that $z$ is adjacent to distinct vertices $a$ and $b$ in $V(C_i)$, and let $N_G(a)=\{c,d\}$. By~\ref{sublem:G--xy is 3-con}, $G\dd za$, which equals $(G-a)+cd$, is $3$-connected. If $G\dd za$ has a proper $3$-connected subgraph $D$, then $D$ must contain $cd$, otherwise $D$ is a proper $3$-connected subgraph of $G$. Without loss of generality, $c\in V(C_i)$, so there is a path of degree-$3$ vertices containing $b$ and $c$ in $G\dd za$. Therefore, $D$ must contain $b$ as well as $z$. Let $D'$ be the graph obtained from $D$ by bridging $\{z,cd\}$. It is straightforward to see that $D'$ is a proper $3$-connected subgraph of $G$, a contradiction. Therefore, $G\dd za$ is super-minimally $3$-connected. Thus,
\[|E(G)|=|E(G\dd za)|+2\leq 2|V(G\dd za)|=2|V(G)|-2.\]
 Because $|E(G)|\geq 2|V(G)|-2$, we know $|E(G\dd za)|=2|V(G\dd za)|-2$ and, by the inductive hypothesis, $G\dd za$ is a wheel on at least three spokes. By Lemma~\ref{lem:briding a wheel}, we know $G$ is a wheel, a contradiction. Thus, there is at most one edge between $z$ and $V(C_i)$. Moreover, since $Z\neq\emptyset$, we know $k\geq 4$. Consequently, $|Z|\geq 3$, otherwise $Z$ is a $1$-cut or $2$-cut. Thus~\ref{sublem:one edge} holds.\\

\begin{sublemma}\label{sublem:tree}
    For all $i\in\{1,2,\dots,k\}$, the graph $C_i$ is a tree.
\end{sublemma}

Suppose that $C_i$ has a cycle $a_1a_2\dots a_ta_1$. Let $N_G(a_1)=\{a_2,a_t,b_1\}$ and let $N_G(a_2)=\{a_1,a_3,b_2\}$. Since $G$ is triangle-free, the vertices $a_3,a_t,b_1$, and $b_2$ are distinct. By~\ref{sublem:G--xy is 3-con}, $G\dd a_1a_2$, which equals $(G-\{a_1,a_2\})+b_1a_t+b_2a_3$, is $3$-connected. If $G\dd a_1a_2$ has a proper $3$-connected subgraph $D$, then $D$ must contain $b_1a_t$ or $b_2a_3$, otherwise $D$ is a proper $3$-connected subgraph of $G$. However, because $a_3a_4\dots a_t$ is a path of degree-$3$ vertices in $G\dd a_1a_2$, we know $D$ contains both $a_3$ and $a_t$, and hence contains both $b_1a_t$ and $b_2a_3$. Let $D'$ be the graph obtained from $D$ by bridging $\{b_1a_t,b_2a_3\}$. It is straightforward to see that $D'$ is a proper $3$-connected subgraph of $G$, a contradiction. Therefore, $G\dd a_1a_2$ is super-minimally $3$-connected. However,
\[|E(G)|=|E(G\dd a_1a_2)|+3\leq 2|V(G\dd a_1a_2)|+1=2|V(G)|-3,\]
a contradiction. Therefore~\ref{sublem:tree} holds.\\

Let $G'=G/(E(C_1)\cup E(C_2)\cup\dots\cup E(C_k))$. By~\ref{sublem:one edge} and~\ref{sublem:tree}, we know $G'$ is simple. For each $i\in\{1,2,\dots,k\}$, let $c_i$ be the vertex obtained by identifying all the vertices in $C_i$, and let $C=\{c_1,c_2,\dots,c_k\}$. By~\ref{sublem:all incident to deg-3}, $Z$ is a stable set in $G$, so $G'$ is a bipartite graph with bipartition $(Z,C)$ such that $d_{G'}(c)\geq 3$ for all $c\in C$. Moreover, by~\ref{sublem:tree},
\[
d_{G'}(c_i) \;=\; 3|V(C_i)| - 2|E(C_i)| \;=\; |E(C_i)| + 3,
\]
for all $i \in \{1,2,\dots,k\}$. In particular,
\[
\sum_{i=1}^k d_{G'}(c_i)
   \;=\; \sum_{i=1}^k \bigl(|E(C_i)| + 3\bigr).\tag{8}\label{eq:eight}
\]

Next, we show that

\begin{sublemma}\label{sublem:k<=2h-3}
    $k\leq 2|Z|-3$.
\end{sublemma}

Suppose $k \geq 2|Z|-2$. By Corollary~\ref{cor:y>= 2x-2}, $C$ has a proper subset $C'$ such that $G'[N_{G'}[C']]$ is $3$-connected. 
Let $\cC = \{\, C_i \in \{C_1,C_2,\dots,C_k\} : c_i \in C'\,\}$ and $\cV = N_{G'}(C') \cup \{\, v \in V(C_i) : C_i \in \cC \,\}$. 
Note that $G'[N_{G'}[C']] = G[\cV] / \{\, e \in E(C_i) : C_i \in \cC \,\}$. 
Let $u$ be a vertex in $N_{G'}(C')$. Because none of the edges in $\{\, e \in E(C_i) : C_i \in \cC \,\}$ is incident to $u$ in $G[\cV]$ and $G'[N_{G'}[C']]$ is $3$-connected, we know $d_{G[\cV]}(u)=d_{G'[N_{G'}[C']]}(u)\geq 3$.
Moreover, for each vertex $c \in V(C_i)$ with $C_i \in \cC$, because every neighbor of $c$ in $G$ is contained in $\cV$, we know $d_{G[\cV]}(c)=d_G(c)=3$. 
By~\ref{sublem:tree}, $G'[N_{G'}[C']]$ is obtained from $G[\mathcal{V}]$ by contracting a forest and, by~\ref{sublem:one edge}, $G'[N_{G'}[C']]$ is simple. Because $G'[N_{G'}[C']]$ is $3$-connected and the minimum degree of $G[\mathcal{V}]$ is at least three, by Lemma~\ref{lem:contracting forest}, $G[\mathcal{V}]$ is $3$-connected.
Therefore, $G[\cV]$ is a proper $3$-connected subgraph of $G$, a contradiction.  Thus~\ref{sublem:k<=2h-3} holds.\\

From~\ref{sublem:one edge},~\ref{sublem:tree}, and~\ref{sublem:k<=2h-3}, we obtain the following:
{\allowdisplaybreaks
\begin{align*}
    |E(G)|&=|E(G')|+\sum_{i=1}^k|E(C_i)|\\
    &=\sum_{i=1}^k d_{G'}(c_i)+\sum_{i=1}^k|E(C_i)|\\
    &=\sum_{i=1}^k(|E(C_i)|+3)+\sum_{i=1}^k|E(C_i)| && \text{by~(\ref{eq:eight}),}\\
    &=3k+2\sum_{i=1}^k|E(C_i)|\\
    &\leq 2k+2|Z|-3+2\sum_{i=1}^k|E(C_i)|&& \text{by~(\ref{sublem:k<=2h-3}),}\\
    &=2|Z|-3+2\sum_{i=1}^k(|E(C_i)|+1)\\
    &=2|Z|-3+2\sum_{i=1}^k|V(C_i)|&& \text{by~(\ref{sublem:tree}),}\\
    &=2|V(G)|-3.
\end{align*}
}
This contradiction completes the proof.
\end{proof}

\section*{Acknowledgments}
The author sincerely thanks James Oxley for proposing the study of super-minimally $3$-connected graphs. His invaluable guidance and insightful advice were essential to this work.

\end{document}